\documentclass[11pt,reqno]{amsart}
\usepackage[utf8]{inputenc}
\usepackage{amsmath}
\usepackage{wasysym} 
\usepackage{graphicx}
    \graphicspath{{./figures/}}
\usepackage{color}

\usepackage{pgf}
\usepackage{enumitem}
\usepackage{algorithm}
\usepackage{setspace}
\usepackage{algpseudocode}
\usepackage{mathtools}
\usepackage{microtype}
\usepackage{xfrac}
\usepackage{xspace}

\usepackage{accents}
\usepackage{multicol}
\usepackage{multirow}
\usepackage{soul}
\usepackage{circuitikz}

\usepackage{thmtools} 

\usepackage[font=footnotesize,margin=0.2cm,skip=5pt]{caption}
\usepackage{subcaption}

\usepackage{booktabs}
\usepackage{array}
\usepackage{comment}
\newcolumntype{L}[1]{>{\raggedright\let\newline\\\arraybackslash\hspace{0pt}}m{#1}}
\newcolumntype{C}[1]{>{\centering\let\newline\\\arraybackslash\hspace{0pt}}m{#1}}
\newcolumntype{R}[1]{>{\raggedleft\let\newline\\\arraybackslash\hspace{0pt}}m{#1}}
\usepackage{siunitx}
\usepackage{bbold}
\usepackage{bm}

\usepackage[hidelinks]{hyperref}
\usepackage[nameinlink,noabbrev]{cleveref}

\newtheorem{theorem}{Theorem}
\newtheorem{proposition}[theorem]{Proposition}

\theoremstyle{definition}

\newtheorem{example}[theorem]{Example}

\newtheorem{lemma}[theorem]{Lemma}

\newtheorem{remark}[theorem]{Remark}
\newtheorem{assumption}[theorem]{Assumption}
\newtheorem{scheme}[theorem]{Scheme}

\Crefname{assumption}{Assumption}{Assumptions}
\numberwithin{theorem}{section}
\numberwithin{proposition}{section}
\numberwithin{corollary}{section}
\numberwithin{definition}{section}
\numberwithin{example}{section}
\numberwithin{lemma}{section}
\numberwithin{remark}{section}
\numberwithin{assumption}{section}
\numberwithin{scheme}{section}
\numberwithin{conjecture}{section}
\numberwithin{equation}{section}
\numberwithin{table}{section}
\numberwithin{figure}{section}

\definecolor{myBlue}{RGB}{30,144,255} 
\definecolor{myGreen}{RGB}{69,169,0} 
\definecolor{myRed}{RGB}{165,12,42}
\definecolor{myOrange}{RGB}{225,92,22}

\definecolor{mycolor0}{rgb}{0.12156862745098,0.466666666666667,0.705882352941177} 

\definecolor{mycolor1}{rgb}{0.00000,0.44700,0.74100}
\definecolor{mycolor2}{rgb}{0.85000,0.32500,0.09800}
\definecolor{mycolor3}{rgb}{0.49400,0.18400,0.55600}
\definecolor{mycolor4}{rgb}{0.92900,0.69400,0.12500}
\definecolor{mycolor5}{rgb}{0.46600,0.67400,0.18800}
\definecolor{mycolor6}{rgb}{0.30100,0.74500,0.93300}
\definecolor{mycolor7}{rgb}{0.63500,0.07800,0.18400}

\def\N{\mathbb{N}}

\def\R{\mathbb{R}}

\newcommand{\RR}{\R}

\newcommand\ds{\,\mathrm{d}s}
\newcommand\dt{\,\mathrm{d}t}
\newcommand\dx{\,\mathrm{d}x}

\newcommand{\patz}[1]{\partial_{z_{#1}}}
\newcommand{\patc}[1]{\partial_{c^h_{#1}}}
\newcommand{\patzol}[1]{\partial_{\ol{z}_{#1}}}
\newcommand{\patzr}[1]{\partial_{\zr_{#1}}}

\newcommand{\derivative}{D}

\newcommand{\eps}{\ensuremath{\varepsilon}}

\newcommand{\Amat}{\mathbf{A}}
\newcommand{\Bmat}{\mathbf{B}}

\newcommand{\Gmat}{\mathbf{G}}
\newcommand{\Imat}{\mathbf{I}}
\newcommand{\Jmat}{\mathbf{J}}
\newcommand{\Kmat}{\mathbf{K}}
\newcommand{\Mmat}{\mathbf{M}}

\newcommand{\Pmat}{\mathbf{P}}
\newcommand{\Rmat}{\mathbf{R}}

\newcommand{\Vmat}{\mathbf{V}}
\newcommand{\Wmat}{\mathbf{W}}

\newcommand{\jbf}{\mathbf{j}}
\newcommand{\rbf}{\mathbf{r}}

\newcommand{\bbf}{\mathbf{g}}

\newcommand{\jbfr}{\tilde{\mathbf{j}}}
\newcommand{\rbfr}{\tilde{\mathbf{r}}}
\newcommand{\bbfr}{\tilde{\mathbf{g}}}
\newcommand{\yr}{\tilde{y}} 

\newcommand{\ol}{\overline}

\DeclareFontFamily{U}{matha}{\hyphenchar\font45}
\DeclareFontShape{U}{matha}{m}{n}{
	<-6> matha5 <6-7> matha6 <7-8> matha7
	<8-9> matha8 <9-10> matha9
	<10-12> matha10 <12-> matha12
}{}
\DeclareSymbolFont{matha}{U}{matha}{m}{n}

\DeclareFontFamily{U}{mathx}{\hyphenchar\font45}
\DeclareFontShape{U}{mathx}{m}{n}{
	<-6> mathx5 <6-7> mathx6 <7-8> mathx7
	<8-9> mathx8 <9-10> mathx9
	<10-12> mathx10 <12-> mathx12
}{}
\DeclareSymbolFont{mathx}{U}{mathx}{m}{n}

\DeclareMathDelimiter{\vvvert} {0}{matha}{"7E}{mathx}{"17}%

\newcommand{\norm}[1]{\lVert #1 \rVert}
\newcommand{\ham}{\mathcal{H}} 
\newcommand{\poly}{\mathbb{P}} 
\newcommand{\cpoly}{\mathsf{c}\poly} 
\newcommand{\ansatz}{\mathbb{V}} 
\newcommand{\test}{\mathbb{W}} 
\newcommand{\proj}{\Pi} 
\newcommand{\projmor}{\mathbf{\Pi}} 
\newcommand{\interp}{\mathcal{I}} 
\newcommand{\coeff}{\zeta} 
\newcommand{\tauh}{{\tau h}}
\newcommand{\tauhj}{{\tau h,j}} 
\renewcommand{\tilde}{\widetilde}

\DeclareMathOperator{\diag}{diag}
\newcommand{\eye}{\Imat} 

\newcommand{\nr}{\tilde{n}} 
\newcommand{\zr}{\tilde{z}} 
\newcommand{\vr}{\tilde{v}} 
\newcommand{\hamr}{\tilde{\mathcal{H}}} 

\newcommand{\differror}{e_{\mathrm{d}}}
\newcommand{\thetar}{\tilde{\theta}}

\newcommand{\quadnodes}{s_Q} 
\newcommand{\projnodes}{s_\Pi} 

\newcommand{\normal}{\mathbf{n}} 

\newcommand{\pat}{\partial_t} 

\renewcommand{\phi}{\varphi}

\newcommand{\error}{\mathcal{E}}
\newcommand{\errorenergy}{\error_{\text{energy}}}
\newcommand{\errorstate}{\error_{\text{state}}}
\newcommand{\errorstatenodal}{\error_{\text{state,nodal}}}

\newcommand{\zm}{\mathbf{z}}

\newcommand{\manuin}{\mathbf{m}}

\newcommand{\tp}{^{\hspace{-1pt}\mathsf{T}}}
\renewcommand{\vec}[1]{\begin{bmatrix} #1 \end{bmatrix}}
\newcommand{\textvec}[1]{[\begin{smallmatrix} #1 \end{smallmatrix}]}
\newcommand{\bigtextvec}[1]{\Big[\begin{smallmatrix} #1 \end{smallmatrix}\Big]}
\newcommand{\Bigtextvec}[1]{\bigg[\begin{smallmatrix} #1 \end{smallmatrix}\bigg]}
\newcommand{\btvec}{\bigtextvec}
\newcommand{\Btvec}{\Bigtextvec}

\newcommand{\blangle}{\big\langle}
\newcommand{\brangle}{\big\rangle}
\newcommand{\Blangle}{\Big\langle}
\newcommand{\Brangle}{\Big\rangle}

\newcommand{\embeds}{\hookrightarrow}
\newcommand{\boundedlinear}{\mathcal{L}}

\newcommand{\abs}[1]{\lvert #1 \rvert}
\newcommand{\lebesgue}{L} 

\newcommand{\pax}{\partial_x}

\newcommand{\dnpa}{\mathfrak{a}} 
\newcommand{\dnpb}{\mathfrak{b}} 
\newcommand{\dnptimescale}{\theta}

\newcommand{\Ac}{\Amat_{\mathrm{C}}}
\newcommand{\AL}{\Amat_{\mathrm{L}}}
\newcommand{\AR}{\Amat_{\mathrm{R}}}
\newcommand{\AS}{\Amat_{\mathrm{S}}}
\newcommand{\HC}{\ham_{\mathrm{C}}}
\newcommand{\HL}{\ham_{\mathrm{L}}}
\newcommand{\iS}{\imath_{\mathrm{S}}}
\newcommand{\psiL}{\psi_{\mathrm{L}}}
\newcommand{\qC}{q_{\mathrm{C}}}
\newcommand{\uS}{u_{\mathrm{S}}}

\newcommand{\banachspace}{\mathsf{B}}

\allowdisplaybreaks

\makeatletter
\let\c@table\c@figure
\let\ftype@table\ftype@figure
\makeatother

\begin{document}

\title[Structure-preserving Discretization and MOR for Energy-Based Models]{Structure-preserving Discretization and Model Reduction for Energy-Based Models}
\author[]{R.~Altmann$^\dagger$, A.~Karsai$^\ddagger$, P.~Schulze$^\ddagger$}
\address{${}^{\dagger}$ Institute of Analysis and Numerics, Otto von Guericke University Magdeburg, Universit\"atsplatz 2, 39106 Magdeburg, Germany}
\address{${}^{\ddagger}$ Institute of Mathematics, Technische Universität Berlin, Str.~des 17. Juni~136, 10623 Berlin, Germany}
\email{robert.altmann@ovgu.de, karsai@math.tu-berlin.de, pschulze@math.tu-berlin.de}

\date{\today}

\begin{abstract}
    We investigate discretization strategies for a recently introduced class of energy-based models.
    The model class encompasses classical port-Hamiltonian systems, generalized gradient flows, and certain systems with algebraic constraints.
    Our framework combines existing ideas from the literature and systematically addresses temporal discretization, spatial discretization, and model order reduction, ensuring that all resulting schemes are dissipation-preserving in the sense of a discrete dissipation inequality.
    For this, we use a Petrov--Galerkin ansatz together with appropriate projections.
    Numerical results for a nonlinear circuit model, the Cahn--Hilliard equation, and a doubly nonlinear parabolic equation illustrate the effectiveness of the approach.
\end{abstract}

\maketitle


\section{Introduction}
Ordinary and partial differential equations are used to model various physical processes.
To make these mathematical models more accurate, it may be beneficial to explicitly model certain conserved quantities, preventing the numerical approximations to become unphysical.
Incorporating such constraints as well as coupling conditions or constitutive relations leads to additional (algebraic) equations and, hence, to partial differential--algebraic equations, in general.

Exemplary model classes following this paradigm include gradient systems~\cite{HirS74,EggHS21}, the GENERIC framework~\cite{GrmO97}, port-Hamiltonian structures~\cite{Arn89,vanderschaft20-dirac}, and recently introduced adaptations of the latter~\cite{AltS25,GieKT25}.

Numerical simulations require temporal and possibly spatial discretization of the aforementioned model classes.
Unfortunately, classical discretization methods do not take the additional system properties into account, which can lead to physical inconsistencies within the simulation results.
A multitude of structure-preserving discretization and model reduction methods have been developed as a remedy.
The focus of this manuscript is to present a class of discretization methods that enable the qualitative conservation of an energy quantity.
In fact, we are interested in preserving an energy balance and the corresponding dissipation inequality.
We note that this is a specific design choice, especially for time discretization.
Another possibility would be to preserve the geometric structure of the energy flow, leading to symplectic methods, cf.~\cite{HaiWL06}.
Unfortunately, for the case of energy-conservative systems, exact preservation of the energy and symplecticity cannot be achieved at the same time by a numerical scheme with fixed step size~\cite{zhong88-lie}.
In particular, the scheme presented here is not generally symplectic.

To preserve an energy balance in discrete models, suitable approaches include Petrov--Galerkin methods in time~\cite{gross05-conservation,andrews25-enforcing} or space and time~\cite{jackaman19-finite,EggHS21,GieKT25,Mor24}.
For the spatial discretization, mixed finite element techniques have been developed~\cite{liljegren19-structure,serhani19-partitioned,egger23-asymptotic}, which tend to be tailored to the particular problem structure at hand, exploiting the relationship between the discretized variables.
For the temporal discretization, possible approaches include collocation and Runge--Kutta methods~\cite{hairer10-energy,kotyczka19-discrete,MehM19,Mor24}, operator splitting approaches~\cite{CelH17,FroGLM24}, and symplectic integrators as mentioned before.
Moreover, discrete gradient schemes, originally introduced in~\cite{harten83-upstream,Gon96,McLQR99}, were applied to port-Hamiltonian systems in~\cite{CelH17,macchelli23-control,Sch24,kinon25-discrete}.

Concerning model order reduction, Petrov--Galerkin methods are again suitable~\cite{ChaBG16}, including balancing approaches \cite{phillips02-guaranteed,PolS12,breiten22-error} and interpolatory techniques \cite{GugPBS12}.
Also spectral factorization methods~\cite{breiten22-passivity} and optimization approaches~\cite{schwerdtner23-sobmor} have been proposed.
Although we are focused on linear model reduction in this manuscript, we remark that there is increasing attention in nonlinear model reduction techniques as well~\cite{kawano18-structure,LeeC19,peherstorfer22-breaking,BucGHU24}.

Our main contribution is the development of a structure-preserving discretization and model reduction framework for the model class considered in~\cite{AltS25}.
Our approach combines ideas of~\cite{ChaBG16,EggHS21,andrews25-enforcing,GieKT25} and relies on a Petrov--Galerkin method in combination with suitable projections applied to the gradient of the energy functional.
To the best of our knowledge, this idea was first used in~\cite{ChaBG16}, where the authors assumed that this gradient is well-approximated by the test space used in the Petrov--Galerkin method.
The projection based formulation used here first appeared in~\cite{GieKT25}.
The models considered in this paper appeared in a similar form in~\cite{vanderschaft20-dirac} and can be viewed as a combination of the classes from~\cite{EggHS21,GieKT25} with additional algebraic constraints.
Unlike~\cite{andrews25-enforcing}, our model class and discretization scheme account for possible nonlinear effects on the temporal derivatives of the variables and is not restricted to problems with homogeneous or periodic boundary conditions.
On the other hand, we focus on the qualitative conservation of a single quantity.
Another contribution of this paper is the rigorous formulation of the class considered in~\cite{AltS25} for infinite-dimensional systems.
Since the class covers a wide range of physical phenomena, we do not provide existence and uniqueness results for the considered systems.
Furthermore, we show for a special case of our structure how spatial discretization via projection onto on a nonlinear manifold preserves the structure and we provide an error bound for the structure-preserving model reduction scheme for the special case of a semi-explicit system of index one.

The paper is structured as follows.
In \Cref{sec:model}, we formulate the model class of~\cite{AltS25} for infinite-dimensional systems and recall essential properties.
\Cref{sec:discretization} contains the main contribution of this work, a unified framework for the discretization in time and space as well as model reduction.
We first propose a fully discrete scheme and subsequently discuss its properties for the special cases of time continuous spatial discretization, model order reduction, and temporal discretization.
\Cref{sec:numerics} presents three numerical examples, including a nonlinear circuit model, the Cahn--Hilliard equation, and a doubly nonlinear parabolic equation.
Here, we discuss details of the implementation and study the numerical properties of our scheme.
Finally, we summarize our findings in \Cref{sec:conclusion} and point toward possibilities for future research.

\subsubsection*{Notation}
Throughout the paper, we write $(x,y) \coloneqq [x\tp, y\tp]\tp$ for two column vectors~$x,y$.
Moreover, we introduce the empty element~$\bullet$ in order to distinguish it from possible zero entries.
With this, we have, e.g., $(x, \bullet) = x$.
Besides, we use $\diag(\Amat_1,\Amat_2,\ldots,\Amat_k)$ to denote (block) diagonal matrices with diagonal (block) entries $\Amat_1,\ldots,\Amat_k$.
The dual space of the normed space~$X$ is denoted by~$X^*$ and the set of bounded linear operators from~$V$ to~$W$ is denoted by $\boundedlinear(V,W)$.

\section{Energy-based Model}\label{sec:model}

This section is devoted to the introduction of the energy-based framework which rigorously generalizes the one from~\cite{AltS25} to infinite-dimensional systems.
In~\cite{AltS25}, finite-dimensional systems of the form
\begin{subequations}\label{eq:model:altschulz-finite-dim}
    \begin{align}        \vec{
            \patz 1 \ham(z_1, z_2)
            \\
            \dot{z}_2
            \\
            0
        }
        & =
        (\Jmat - \Rmat)
        \vec{
            \dot{z}_1
            \\
            \patz 2 \ham(z_1, z_2)
            \\
            z_3
        }
        +
        \Bmat u,
        \qquad z(0) = z_0,
        \\
        y
        & =
        \Bmat\tp
        \vec{
            \dot{z}_1
            \\
            \patz 2 \ham(z_1, z_2)
            \\
            z_3
        },\end{align}\end{subequations}
with matrices~$\Jmat, \Rmat \in \RR^{n,n}$ with~$\Jmat = -\Jmat\tp$ and~$\Rmat = \Rmat\tp \succeq 0$ and~$\Bmat \in \RR^{n,m}$ are considered, where ``$\succeq$'' is the Loewner order.
Here, the state~$z \in \RR^n$ is split into three parts, i.e.,~$z = (z_1, z_2, z_3) \in \RR^{n_1} \times \RR^{n_2} \times \RR^{n_3}$,~$n_1 + n_2 + n_3 = n$, and the energy~$\ham \colon \RR^{n_1} \times \RR^{n_2} \to \RR$ is assumed to depend on~$z_1$ and~$z_2$ only.
The variable~$u$ denotes an external control input, the variable~$y$ describes a \emph{collocated} output, and $z_0 \in \RR^n$ is an initial condition to be specified.
For the sake of brevity, we suppress the time dependency of the variables~$z_1$,~$z_2$,~$z_3$, and~$u$.
We denote the derivative of~$\ham$ with respect to the first and second variables by~$\patz 1 \ham$ and~$\patz 2 \ham$, respectively.

Due to the hybrid structure as well as the existence of the $z_3$-variable, the framework~\eqref{eq:model:altschulz-finite-dim} is especially suited for differential--algebraic equations. This also includes examples of higher index

(see~\cite{Meh13} for an overview of different index concepts) as shown in~\cite{AltS25}. In particular, the index is not bounded by two as for classical port-Hamiltonian systems~\cite{MehMW18}.

Although the authors of~\cite{AltS25} do not explicitly consider the case of state dependent~$\Jmat$,~$\Rmat$, and~$\Bmat$, they emphasize that the model class can be easily extended to include these cases.
Here, we take a different approach to model nonlinearities in~$\Jmat$ and $\Rmat$.
Instead of using matrices, we replace~$\Jmat$,~$\Rmat$ and~$\Bmat$ by possibly nonlinear functions~$\jbf, \rbf \colon \RR^n \to \RR^n$ satisfying~$v\tp \jbf(v) = 0$ and~$v\tp \rbf(v) \geq 0$ for all~$v\in \RR^n$ as well as~$\bbf \colon \RR^n \to \RR^{n,m}$, where we allow $\bbf$ to be state-dependent.
This leads to the model
\begin{subequations}\label{eq:model-strong}
    \begin{align}        \vec{
            \patz 1 \ham(z_1, z_2)
            \\
            \dot{z}_2
            \\
            0
        }
        & =
        (\jbf - \rbf)
        \vec{
            \dot{z}_1
            \\
            \patz 2 \ham(z_1, z_2)
            \\
            z_3
        }
        +
        \bbf(z) u,
        \qquad z(0) = z_0,
        \\
        y
        & =
        \bbf(z)\tp
        \vec{
            \dot{z}_1
            \\
            \patz 2 \ham(z_1, z_2)
            \\
            z_3
        }.\end{align}\end{subequations}

As we will see later, structure-preserving discretizations of infinite-dimensional generalizations of~\eqref{eq:model-strong} lead to models of this form.

To generalize this class to the infinite-dimensional case, we make the following assumption.

\begin{assumption}\label{as:model:altschulz-assumption}
    Let~$X_1, X_2$ be reflexive Banach spaces and~$Z_1, Z_2, X_3 = Z_3$ be real Hilbert spaces identified with their dual such that there exist continuous and dense embeddings~$X_i \embeds Z_i$ for~$i=1,2$.
    Moreover, consider Banach spaces~$\tilde{X}_1$ and~$\tilde{X}_2$ with continuous and dense embeddings~$\tilde{X}_i \embeds X_i$,~$i=1,2$, and an open subset~$D \subseteq \tilde{X}_1 \times \tilde{X}_2$.
    Let~$\ham \colon D \to \RR$ be continuously Fréchet differentiable.
    Finally, let~$U$ be a normed space.
\end{assumption}

Due to the continuous and dense embeddings in \Cref{as:model:altschulz-assumption}, we consider two Gelfand triples $X_i, Z_i, X_i^*$ for $i=1,2$; see~\cite[Ch.~23.4]{Zei90a}.
To shorten the notation, we introduce
\begin{equation*}    X \coloneqq X_1^* \times X_2 \times X_3.\end{equation*}
In the model equations, we will restrict ourselves to states~$(z_1, z_2) \in X_1 \times X_2$ for which additionally~$\patz i \ham(z_1, z_2) \in X_i$ holds for~$i=1,2$, which is a smoothness assumption on $\ham$.
Let now~$\jbf, \rbf \colon X \to X^*$ be possibly nonlinear operators satisfying
\begin{equation}\label{eq:J-R-properties}
    \langle \jbf(v), v \rangle_{X^*,X} = 0
    ,\qquad
    \langle \rbf(v), v \rangle_{X^*,X} \geq 0\end{equation}
for all~$v \in X$.
Further, similar to the situation before, we consider a continuous function~$\bbf \colon X_1 \times X_2 \times X_3 \to \boundedlinear(U, X^*)$.
The generalization of~\eqref{eq:model:altschulz-finite-dim} now reads
\begin{subequations}\label{eq:model-weak}
    \begin{align}    \Blangle
    \Btvec{\patz 1 \ham(z_1,z_2) \\ \pat z_2 \\ 0}
    ,
    \phi
    \Brangle_{X^*, X}
    & =
    \Blangle
        (\jbf - \rbf) \Btvec{\pat z_1 \\ \patz 2 \ham(z_1,z_2) \\ z_3}
        ,
        \phi
    \Brangle_{X^*, X}
    +
    \blangle
        \bbf(z) u
        ,
        \phi
    \brangle_{X^*,X}
    ,\quad
    z(0) = z_0
    ,
    \\
    y
    & =
    \bbf(z)^* \Btvec{\pat z_1 \\ \patz 2 \ham(z_1,z_2) \\ z_3}\end{align}\end{subequations}
for all test functions~$\phi \in X$ and~$u(\cdot)\in U$.
As expected,~$\jbf$ and~$\rbf$ model conservative and dissipative effects, respectively, and~$\bbf$ handles external inputs to the model.

These inputs can be classical controls, but also boundary conditions that are necessary for the well-posedness of the model can be considered as inputs.
Although we do not consider explicit state dependencies of the operators~$\jbf$,~$\rbf$ and~$\bbf$, we remark that our scheme can be easily extended to these cases.

We complement~\eqref{eq:model-weak} with the initial condition

\[
    (z_1(0), z_2(0), z_3(0))
    = z_0
    = (z_{1,0}, z_{2,0}, z_{3,0}) \in X_1 \times X_2 \times X_3.
\]

Similar to~\eqref{eq:model:altschulz-finite-dim}, equation~\eqref{eq:model-weak} only contains the time derivatives of $z_1$ and $z_2$, where $\pat z_1$ occurs only implicitly in the argument of $\jbf - \rbf$.
Hence, the initial condition for $z_2(0)$ is free of choice, whereas $z_3(0)$ and possibly $z_1(0)$ need to satisfy consistency conditions.

\begin{example}
\label{exp:consistentIC}
Consider $\ham =  \frac12 z_1^2 + \frac12 z_2^2$ and $\jbf = 0$, $\rbf = \diag(0,0,1)$. Then $z_2(0)$ can be chosen freely but $z_1(0)$ and $z_3(0)$ are fixed by the input at time $t=0$.
\end{example}

Since different physical models can be written in the form~\eqref{eq:model-weak}, a detailed discussion of the existence of (unique) solutions is not in the scope of this article.
Instead, we assume that $u$ is given such that~\eqref{eq:model-weak} admits a unique solution on the time interval $[0,T]$ with $T > 0$.

Formulation~\eqref{eq:model-weak} covers a wide range of physical models.
This includes the quasilinear wave equation, 
the viscoelastic Stokes problem, the Cahn--Hilliard equation, poroelasticity, power network and circuit models, and constrained mechanical systems; see~\cite{AltS25,AltCPSS26}.
It also covers the $p$-Laplacian~\cite{GieKT25} and the equations of magneto-quasistatics~\cite{EggHS21}.

\subsection{Key properties}

One important property of solutions of~\eqref{eq:model-weak} are the following energy balance and dissipation inequality.

\begin{proposition}\label{prop:energy-balance}
    Sufficiently smooth solutions $z=(z_1,z_2,z_3)$ of~\eqref{eq:model-weak} satisfy
    \begin{align}        \ham(z_1(\theta_2), z_2(\theta_2)) - \ham(z_1(\theta_1), z_2(\theta_1))
        &= \int_{\theta_1}^{\theta_2}
        -\, \Big\langle
        \rbf \bigtextvec{\pat z_1 \\ \patz 2 \ham(z_1,z_2) \\ z_3}
        ,
        \bigtextvec{\pat z_1 \\ \patz 2 \ham(z_1,z_2) \\ z_3}
        \Big\rangle _{X^*, X} \label{eq:energy-balance}
        \\*
        &
        \hphantom{= \int_{\theta_1}^{\theta_2}}
        +
        \langle
            y, u
            \rangle _{U^*,U}
        \dt \notag
        \\
        &\leq
        \int_{\theta_1}^{\theta_2}
        \langle
            y, u
            \rangle _{U^*,U}
        \dt
        \label{eq:dissipation}\end{align}
    for all $0 \le \theta_1 \le \theta_2 \le T$ and $u\colon [0,T] \to U$.
\end{proposition}

\begin{proof}
Similarly as in~\cite{AltS25}, where the strong form~\eqref{eq:model-strong} was considered, the equality~\eqref{eq:energy-balance} is obtained as a direct consequence of the properties of $\jbf$.
The inequality~\eqref{eq:dissipation} follows from the properties of $\rbf$.
\end{proof}

Furthermore, the structure~\eqref{eq:model-weak} is preserved under power-conserving and dissipative interconnections if the state dependence in the input operator of both systems can be expressed in terms of $\ham'(z)$.

\subsection{Relation to other model classes}
Let us shortly discuss how model~\eqref{eq:model-weak} is related to other models from the literature.
Clearly, the following well-known model frameworks are included as special cases.

\begin{example}[port-Hamiltonian systems]
For $z_1 = z_3 = \bullet$, the class~\eqref{eq:model-weak} simplifies to the energy based model class discussed in~\cite{GieKT25}, and \emph{port-Hamiltonian} systems~\cite{Sch13} with state-independent structure and dissipation operators are recovered if $\jbf$ and $\rbf$ are linear; cf.~\cite{AltS25}.
The reader is referred to~\cite{karsai24-nonlinear} for a detailed discussion on the relationship of finite-dimensional port-Hamiltonian systems and the model class considered in~\cite{GieKT25}.
\end{example}

\begin{example}[gradient systems]
For $z_2 = z_3 = \bullet$, gradient systems with similarities to~\cite{EggHS21} are recovered.
In contrast to~\cite{EggHS21}, we allow nonlinear actions on $\pat z_1$, but do not consider $\jbf$ and $\rbf$ to be state-dependent.
\end{example}

The next two remarks illustrate that one of the variables $z_1$ and $z_2$ can be eliminated from model~\eqref{eq:model-weak}.
This procedure, however, comes at the cost of doubling the corresponding system dimensions.
After removing the $z_2$ variable, the resulting description is similar to the Dirac-dissipative formulation used in~\cite{Mor24}.

\begin{remark}\label{rem:removing-z2}
    To remove the variable~$z_2$ from the model~\eqref{eq:model-weak}, define $\ol{X} \coloneqq X_1^* \times X_2^* \times X_2 \times X_3$ and introduce the variables~$\ol{z}_1 \coloneqq (z_1, z_2) \in X_1 \times X_2$, $\eta_2\vcentcolon= \patz 2 \ham(z_1, z_2)\in X_2$, $\ol{z}_3 \coloneqq (\eta_2, z_3) \in X_2 \times X_3$ and $\ol{\ham}(\ol{z}_1) \coloneqq \ham(z_1, z_2)$.
    We can then write the system~\eqref{eq:model-weak} in the form
    \begin{align*}        (\patzol 1 \ol{\ham}(\ol{z}_1) )_1
        & =
        \patz 1 \ham(z_1, z_2)
        =
        (\jbf - \rbf)
        \Btvec{
            \pat z_1 \\ \eta_2 \\ z_3
            }_1
        + \big(\bbf(z) u\big)_1
        =
        (\jbf - \rbf)
        \Btvec{
            \pat (\ol{z}_1)_1 \\ (\ol{z}_3)_1 \\ (\ol{z}_3)_2
            }_1,
        + \big(\bbf(\ol{z}) u\big)_1,
        \\*
        (\patzol 1 \ol{\ham}(\ol{z}_1) )_2
        &
        =
        \patz 2 \ham(z_1, z_2)
        = \eta_2,
        \\
        0
        & =
        - \pat z_2 + (\jbf - \rbf)
        \Btvec{
            \pat z_1 \\ \eta_2 \\ z_3
        }_2 + \big(\bbf(\ol{z}) u\big)_2
        =
        - \pat (\ol{z}_1)_2
        + (\jbf - \rbf)
        \Btvec{
            \pat (\ol{z}_1)_1 \\ (\ol{z}_3)_1 \\ (\ol{z}_3)_2
        }_2 + \big(\bbf(\ol{z}) u\big)_2,
        \\*
        0
        & =
        (\jbf - \rbf) \Btvec{\pat z_1 \\ \eta_2 \\ z_3}_3
        + \big(\bbf(\ol{z}) u\big)_3
        =
        (\jbf - \rbf)
        \Btvec{
            \pat (\ol{z}_1)_1 \\ (\ol{z}_3)_1 \\ (\ol{z}_3)_2
        }_3
        + \big(\bbf(\ol{z}) u\big)_3,\end{align*}
    where we ignore test functions for the sake of brevity and with the index~$i$ denoting the $i$-th block.
    The above system can equivalently be written as
    \begin{align*}        \Blangle
        \btvec{
            \patzol 1 \ol{\ham}(\ol{z}_1)
            \\
            0
            },
            \phi
        \Brangle _{\ol{X}^*, \ol{X}}
        & =
        \Blangle
            (\ol{\jbf} - \ol{\rbf})
            \btvec{
                \pat \ol{z}_1
                \\
                \ol{z}_3
            }
            ,
            \phi
        \Brangle _{\ol{X}^*, \ol{X}}
        +
        \big\langle
            \ol{\bbf}(\ol{z}) u
            ,
            \phi
        \big\rangle _{\ol{X}^*, \ol{X}},\end{align*}
    where
    \begin{align*}        \ol{\jbf}
        &
        \colon
        \ol{X} \to \ol{X}^*
        , ~~
        \big((a,b),(c,d)\big)
        \mapsto
        \big(
            \jbf(a,c,d)_1, c, -b + \jbf(a,c,d)_2, \jbf(a,c,d)_3
        \big),
        \\
        \ol{\rbf}
        &
        \colon
        \ol{X} \to \ol{X}^*
        , ~~
        \big((a,b),(c,d)\big)
        \mapsto
        \big(
            \rbf(a,c,d)_1, 0, \rbf(a,c,d)_2, \rbf(a,c,d)_3
        \big),
        \\
        \ol{\bbf}
        &
        \colon
        \ol{X} \to \boundedlinear(U, \ol{X}^*)
        , ~~
        \big((a,b),(c,d)\big)
        \mapsto
        \Big(
            u
            \mapsto
            \big(
            (\bbf(a,b,d) u)_1, 0, (\bbf(a,b,d) u)_2, (\bbf(a,b,d) u)_3
            \big)
        \Big).\end{align*}
    Then~$\ol{\jbf}$ and~$\ol{\rbf}$ satisfy~\eqref{eq:J-R-properties}.
    Together with the output $\ol{y} = \ol{\bbf}(\ol{z})^* \textvec{ \pat \ol{z}_1 \\ \ol{z}_3} = \bbf(z)^* \btvec{ \pat z_1 \\ \patz 2 \ham(z_1, z_2) \\ z_3} = y$, we arrive at a system in the form~\eqref{eq:model-weak} but without the $z_2$ variable.
\end{remark}

\begin{remark}
One may also remove the variable~$z_1$ from the model~\eqref{eq:model-weak}. For this, one can follow an idea similarly presented in~\cite{Mor24} and consider the new variables $\ol{z}_2 \coloneqq (z_1, z_2) \in X_1 \times X_2$ and $\ol{z}_3 \coloneqq (\pat z_1, z_3) \in X_1^* \times X_3$.

\end{remark}

\section{Unified Framework for Discretization and Model Reduction}\label{sec:discretization}

Our discretization scheme uses a Petrov--Galerkin ansatz in time and space.
The scheme includes linear model reduction as a special case, see \Cref{subsec:model-reduction}.

\subsection{Trial and test spaces}

For the spatial discretization, let $X_i^h \subseteq X_i$, $i=1,2,3$, be arbitrary finite-dimensional subspaces.
These typically stem from an appropriate discretization using finite elements, but also other choices are possible.
Note that our scheme does not make any assumptions on the structure of the spaces $X_i^h$.
This is in contrast to application-tailored mixed finite element methods, which exploit additional system structures to recover properties like passivity~\cite{liljegren19-structure,serhani19-partitioned,egger23-asymptotic}.
Although mixed finite elements are suitable for our scheme as well, they are not required to obtain an energy balance and the corresponding dissipation inequality.

For the temporal discretization, we assume that the time horizon $[0,T]$ is decomposed at the time points $0 = t_0 < \dots < t_m = T$, $m\in \N$.
We denote the maximum step size by $\tau \coloneqq \max_{j=1,\dots,m} t_j - t_{j-1}$ and use piecewise polynomial spaces for the discretization in time.
Following~\cite{GieKT25}, for a Banach space $\banachspace$ and $k\in \N$, we set
\begin{align*}    \poly_k^\tau(\banachspace)
    & \coloneqq
    \big\{
        p \in L^\infty(0,T; \banachspace)
        ~|~
        p|_{[t_{j-1}, t_{j}]} \text{ is a polynomial of degree $k$}\\*
    &\hspace{4.5cm}   \text{	{with values in $\banachspace$} for all } j=1,\dots,m
    \big\}
    \\*
    \cpoly_k^\tau(\banachspace)
    & \coloneqq
    C(0,T; \banachspace) \cap \poly_k^\tau(\banachspace).\end{align*}
As ansatz spaces for the variables $z_1$, $z_2$, and $z_3$, we consider
\begin{equation*}	\ansatz_1^\tauh \coloneqq \cpoly_k^\tau(X_1^h),
	\qquad
	\ansatz_2^\tauh \coloneqq \cpoly_k^\tau(X_2^h),
	\qquad
	\ansatz_3^\tauh \coloneqq \poly_{k-1}^\tau(X_3^h).\end{equation*}
For the test spaces, we take
\begin{equation*}	\test_1^\tauh \coloneqq \poly_{k-1}^\tau(X_1^h),
	\qquad
	\test_2^\tauh \coloneqq \poly_{k-1}^\tau(X_2^h),
	\qquad
	\test_3^\tauh \coloneqq \poly_{k-1}^\tau(X_3^h)\end{equation*}
and highlight that $\test_i^\tauh = \pat \ansatz_i^\tauh$ for $i=1,2$.
For later use we set
\[
	\ansatz^\tauh \coloneqq \ansatz_1^\tauh \times \ansatz_2^\tauh \times\ansatz_3^\tauh
	\qquad\text{and}\qquad
	\test^\tauh \coloneqq \test_1^\tauh \times \test_2^\tauh \times \test_3^\tauh.
\]

\subsection{Projections}
In order to state a dissipation preserving scheme similar to~\cite{EggHS21,andrews25-enforcing,GieKT25}, we need to include appropriate projections in space and time.
We begin with the projection in space.

For $i=1,2,3$, let $\proj_i^h\colon Z_i \to X_i^h$ be the $Z_i$-orthogonal projection characterized by the property

\begin{equation}\label{eq:spatial-projection-property}
    \langle v, w\rangle _{Z_i}
    =
    \langle v, \proj_i^h w\rangle _{Z_i}
    \quad
    \text{for all $v \in X_i^h$}.\end{equation}

Since $X_i^h$ is finite-dimensional with $X_i^h \subseteq X_i \embeds Z_i$, it is closed with respect to the norm induced by the inner product of $Z_i$ and, thus, the projection~$\proj_i^h \colon X_i \to Z_i$ is well-defined and stable by standard arguments~\cite[Thm.~4.11]{rudin87-real}.
The stability of $\proj_i^h \colon X_i \to X_i$ can be established under additional assumptions on the space~$X_i^h$, see, e.g.,~\cite{ern21-finite1} and the references therein.

As usual but with slight abuse of notation, the projections naturally extend to operators $\proj_i^h\colon L^2(0,T;Z_i) \to L^2(0,T;X_i^h)$.
Moreover, the projection $\proj_i^h$ can be evaluated on~$L^2(0,T;X_i)$.
For later use, we introduce the combined projection

\[
    \proj^h
    \colon Z_1 \times Z_2 \times X_3 \to X_1^h \times X_2^h \times X_3^h,\qquad
    (z_1,z_2,z_3) \mapsto (\proj_1^h z_1, \proj_2^h z_2, \proj_3^h z_3).
\]

The $L^2$-orthogonal temporal projection $\proj_2^\tau\colon L^2(0,T; Z_2)\to \poly_{k-1}^\tau(Z_2)$ is characterized by the property
\begin{equation}\label{eq:temporal-projection-property}
    \int_{0}^{T}
        \langle
            v
            ,
            w
        \rangle _{Z_2}
    \dt
    =
    \int_{0}^{T}
        \langle
            v
            ,
            \proj_2^\tau w
        \rangle _{Z_2}
    \dt
    \quad
    \text{for all $v \in \poly_{k-1}^\tau(Z_2)$}\end{equation}

and $w \in L^2(0,T; Z_2)$.
Note that $\proj_2^\tau$ can be evaluated on $L^2(0,T;X_2)$ and maps $L^2(0,T;X_2) \to \poly_{k-1}^\tau(X_2)$, see~\cite[Lem.~2.1]{GieKT25} for corresponding stability estimates.

Before we state our discretization scheme, we emphasize a key property of the combined projection $\proj_2^\tau \proj_2^h$.%
\begin{lemma}[projection $\proj_2^\tau \proj_2^h$]\label{lem:projection-properties}
    Within the given setting, for any $v \in \ansatz_2^\tauh$ and $w \in L^2(0,T; X_2)$ it holds that
    \begin{equation}\label{eq:projection-property}
        \int_{0}^{T}
            \langle
                \pat v
                ,
                w
            \rangle _{X_2^*, X_2}
        \dt
        =
        \int_{0}^{T}
            \langle
                \pat v
                ,
                \proj_2^\tau \proj_2^h w
            \rangle _{X_2^*, X_2}
        \dt.\end{equation}

\end{lemma}

\begin{proof}
    Note that $X_2 \embeds Z_2 \cong Z_2^* \embeds X_2^*$ implies $\langle \tilde{v}, \tilde{w} \rangle _{X_2^*, X_2} = \langle \tilde{v}, \tilde{w} \rangle _{Z_2}$ for all $(\tilde{v},\tilde{w}) \in Z_2 \times X_2$.
    The claim hence follows from~\eqref{eq:spatial-projection-property} and~\eqref{eq:temporal-projection-property}.

\end{proof}

Equivalently, we can replace the dual pairing $\langle\, \cdot\,, \cdot\,\rangle _{X_2^*, X_2}$ in~\eqref{eq:projection-property} by the inner product $\langle\, \cdot\,, \cdot\,\rangle _{Z_2}$.
We refrain from doing so as the projections will naturally appear in this context later.

\subsection{Numerical scheme}
We are now ready to state the unified discretization scheme.

\begin{scheme}[unified discretization framework]\label{scheme:full-discretization}
Find

\[
    z^\tauh
    = (z_{1}^\tauh, z_{2}^\tauh, z_{3}^\tauh)
    \in \ansatz^\tauh
\]

such that~$\patz i \ham(z_{1}^\tauh, z_{2}^\tauh) \in \lebesgue^2([0,T], X_i)$,~$i=1,2$, and
\begin{equation}\label{eq:model-discrete}
\begin{multlined}
    \int_{0}^{T}
    \Big\langle
        \Bigtextvec{\patz 1 \ham(z_{1}^\tauh,z_{2}^\tauh) \\ \pat z_{2}^\tauh \\ 0}
        ,
        \phi^\tauh
    \Big\rangle _{X^*, X}
    \dt
    \\
    \qquad=
    \int_{0}^{T}
    \Big\langle
        (\jbf - \rbf)
        \Bigtextvec{
            \pat z_{1}^\tauh
            \\
            \proj_2^\tau \proj_2^h\, \patz 2 \ham(z_{1}^\tauh, z_{2}^\tauh)
            \\
            z_{3}^\tauh
            }
        ,
        \phi^\tauh
    \Big\rangle _{X^*, X}
    +
    \big\langle
        \bbf(z^\tauh) u
        ,
        \phi^\tauh
    \big\rangle _{X^*,X}
    \dt
\end{multlined}\end{equation}
for all test functions $\phi^\tauh \in \test^\tauh$.
Moreover, we consider initial conditions for $z^\tauh(0)$, e.g., given by $z^\tauh(0) = \proj^h z_0$.
Note that the initial data may again satisfy a consistency condition as in the continuous setting; cf.~Example~\ref{exp:consistentIC}.
\end{scheme}

For details on when the terms in \Cref{scheme:full-discretization} may be evaluated, we refer to 
\cite{GieKT25,Kar26}.
We mention, however, that if~$z^\tauh$ satisfies~$\patz 1 \ham(z_{1}^\tauh, z_{2}^\tauh) \in \lebesgue^2([0,T], X_1)$ and~$\patz 2 \ham(z_{1}^\tauh, z_{2}^\tauh) \in \lebesgue^2([0,T], X_2)$, then by~\cite[Lem.~2.1]{GieKT25} the projection~$\proj_2^\tau \patz 2 \ham(z_{1}^\tauh, z_{2}^\tauh)$ is well-defined.

\begin{remark}
    Since the test spaces~$\test_i^\tauh$ contain discontinuous functions that are nonzero only in one of the intervals $[t_{j-1}, t_{j}]$, we can localize the integrals in~\eqref{eq:model-discrete}, leading to a time-stepping scheme; see~\cite{Kar26}.

	Note that one obtains $\dim(\test^\tauhj) = (n^h_1 + n^h_2 + n^h_3) \cdot k$ nonlinear equations on each time interval.

    For the existence and uniqueness of solutions to the resulting nonlinear system, we refer to the related approaches~\cite{andrews25-enforcing,ern21-finite3,holm18-continuous}, see also~\cite{Kar26} for a more detailed overview on available results.
\end{remark}

\begin{remark}
    If~$\jbf$ and~$\rbf$ are linear, then the projection~$\proj_2^\tau$ in~\eqref{eq:model-discrete} can be omitted, cf.~\cite{Kar26}.
    If additionally~$z_1$ and~$z_3$ are not present in~\eqref{eq:model-weak} and~$\bbf$ does not depend on the state, then \Cref{scheme:full-discretization} is equivalent to an implicit Runge--Kutta method and can be interpreted as a collocation method, see~\cite[Sect.~70.1.3 and 70.1.4]{ern21-finite3}.

	Recall that the original model may contain differential--algebraic equations such that a thoughtful choice of the time stepping method is needed. In particular, it may happen that solution components converge with different orders~\cite{HaiLR89}.
\end{remark}

\begin{remark}\label{rem:andrews-farrel-approach}
    As mentioned before, in contrast to the related work~\cite{andrews25-enforcing}, \Cref{scheme:full-discretization} accounts for possible nonlinear effects on~$\pat z_1$.
    However, if such nonlinear actions are not present in the model and a reformulation along the lines of \Cref{rem:removing-z2} is peformed, then the use of \Cref{scheme:full-discretization} yields the discretization approach from~\cite{andrews25-enforcing}.
    In this case, our method can be viewed as a special case.
\end{remark}

In the following, we show that discrete solutions coming from ~\Cref{scheme:full-discretization} inherit the energy balance and dissipation inequality from \Cref{prop:energy-balance}.
As mentioned before, since the test functions $\phi^\tauh \in \test^\tauh$ are piecewise continuous, the energy balance and dissipation inequality hold on all discretization intervals $[t_{j-1},t_j]$, $j=1,\dots,m$.

\begin{proposition}\label{prop:energy-balance-discrete}
Let $z^\tauh = (z_1^\tauh, z_2^\tauh, z_3^\tauh)$ denote the solution of \Cref{scheme:full-discretization}.
Then for $j=1,\dots,m$ we have
\begin{align}	\ham(z_1^\tauh(t_{j}), & z_2^\tauh(t_{j})) - \ham(z_1^\tauh(t_{j-1}), z_2^\tauh(t_{j-1})) \nonumber
    \\
    &
    \begin{aligned}
    & = \int_{t_{j-1}}^{t_{j}}
    -\,
    \Big\langle
        \rbf
        \Bigtextvec{
            \pat z_{1}^\tauh
            \\
            \proj_2^\tau \proj_2^h \patz 2 \ham(z_{1}^\tauh, z_{2}^\tauh)
            \\
            z_{3}^\tauh
        }
        ,
        \Bigtextvec{
            \pat z_{1}^\tauh
            \\
            \proj_2^\tau \proj_2^h \patz 2 \ham(z_{1}^\tauh, z_{2}^\tauh)
            \\
            z_{3}^\tauh
        }
    \Big\rangle _{X^*, X}
    +
    \langle
        y^\tauh
        ,
        u
    \rangle _{U^*,U}
    \dt
    \end{aligned}
    \label{eq:energy-balance-discrete}
    \\
    & \leq
    \int_{t_{j-1}}^{t_{j}}
    \langle
        y^\tauh
        ,
        u
    \rangle _{U^*,U}
    \dt,
    \label{eq:dissipation-discrete}\end{align}
where the discrete output~$y^\tauh$ is defined as~$y^\tauh \coloneqq \bbf(z^\tauh)^* \Btvec{\pat z_{1}^\tauh \\ \proj_2^\tau \proj_2^h \patz 2 \ham(z_{1}^\tauh, z_{2}^\tauh) \\ z_{3}^\tauh}$.

\end{proposition}

\begin{proof}
    Following \Cref{lem:projection-properties} and using the fact that $\test_2^\tauh$ contains discontinuous functions, we can localize~\eqref{eq:projection-property}

    and obtain
    \begin{align*}    	\ham(z_1^\tauh(t_{j}), z_2^\tauh(t_{j})) - \ham&(z_1^\tauh(t_{j-1}), z_2^\tauh(t_{j-1}))
        \\*
    &= \int_{t_{j-1}}^{t_{j}}
        \Big\langle
            \Bigtextvec{
                \patz 1 \ham(z_{1}^\tauh,z_{2}^\tauh)
                \\
                \pat z_{2}^\tauh
                \\
                0
            }
            ,
            \Bigtextvec{
                \pat z_{1}^\tauh
                \\
                \patz 2 \ham(z_{1}^\tauh, z_{2}^\tauh)
                \\
                z_{3}^\tauh
            }
        \Big\rangle _{X^*, X}
        \dt
    \\
    &= \int_{t_{j-1}}^{t_{j}}
        \Big\langle
            \Bigtextvec{
                \patz 1 \ham(z_{1}^\tauh,z_{2}^\tauh)
                \\
                \pat z_{2}^\tauh
                \\
                0
            }
            ,
            \Bigtextvec{
                \pat z_{1}^\tauh
                \\
                \proj_2^\tau \proj_2^h \patz 2 \ham(z_{1}^\tauh, z_{2}^\tauh)
                \\
                z_{3}^\tauh
            }
        \Big\rangle _{X^*, X}
        \dt.\end{align*}

    Now, using $\pat z_{1}^\tauh \in \pat \ansatz_1^\tauh = \test_1^\tauh$, $\proj_2^\tau \proj_2^h \patz 2 \ham(z_{1}^\tauh, z_{2}^\tauh) \in \pat \ansatz_2^\tauh = \test_2^\tauh$ and applying~\eqref{eq:model-discrete}, we obtain~\eqref{eq:energy-balance-discrete}.

    Inequality~\eqref{eq:dissipation-discrete} follows from property~\eqref{eq:J-R-properties} of $\rbf$.
\end{proof}

In order to obtain the time-discrete energy balance, we do not need to include projections $\proj_1^\tau \proj_1^h$ in front of $\patz 1 \ham(z_1^\tauh, z_2^\tauh)$ in~\eqref{eq:model-discrete}.
This was also noted in~\cite{EggHS21} and is caused by the fact that the projections are implicitly applied by the first component of the test function~$\phi^\tauh$.

Although we suggest to use piecewise polynomials in time, our scheme can be generalized to other approximations as outlined in the following remark.

\begin{remark}\label{rem:extension-more-general-ansatz-and-test-spaces}
    To obtain the discrete energy balance~\eqref{eq:energy-balance-discrete}, it is crucial to include the projection onto the space $\pat \ansatz_2^\tauh$ in front of $\patz 2 \ham(z_1^\tauh, z_2^\tauh)$.
    This motivates the choice $\test_2^\tauh = \pat \ansatz_2^\tauh$.
    Similarly, choosing $\test_1^\tauh = \pat \ansatz_1^\tauh$ is natural.
    If $\ansatz_i^\tauh$ are chosen as piecewise polynomial spaces as outlined above, the corresponding discretization scheme is well-defined.
    For more general temporal approximations, multiple difficulties arise.
    First, the projection $\proj_2^\tau \proj_2^h$ is generally only well-defined if $\pat \ansatz_2^\tauh$ is a closed subspace of $L^2(0,T; Z_2)$, see, e.g.,~\cite[Thm.~12.4 and Thm.~12.14]{rudin91-functional}.
    Second, the choices $\test_i^\tauh = \pat \ansatz_i^\tauh$, $i=1,2$, do not generally imply 
   	that the resulting nonlinear systems are square, which can lead to difficulties regarding existence and uniqueness of solutions.
    Third, identity~\eqref{eq:projection-property} can no longer be localized in time so that \Cref{prop:energy-balance-discrete} only holds between the boundary points $0$ and $T$.
    Nevertheless, other ansatz spaces may be suitable. For instance, enforcing continuity of $z_3$ may be beneficial.

    Other examples are trigonometric and exponential ansatz functions~\cite{yuan06-discontinuous}.
\end{remark}

To illustrate the implications of the previously presented discretization scheme, we continue with a detailed discussion of the spatial discretization and model reduction as special cases of \Cref{scheme:full-discretization}. We also comment on aspects for practically realizing the time stepping.
For notational reasons, we focus on the time-continuous formulation and drop the time dependence of the test functions for the former two.
This replaces model~\eqref{eq:model-discrete} by
\begin{equation}\label{eq:time-continuous-spatially-discrete}
    \Big\langle
        \Bigtextvec{\patz 1 \ham(z_{1}^h,z_{2}^h) \\ \pat z_{2}^h \\ 0}
        ,
        \phi^h
    \Big\rangle _{X^*, X}
    =
    \Big\langle
        (\jbf - \rbf)
        \Bigtextvec{
            \pat z_{1}^h
            \\
            \proj_2^h \patz 2 \ham(z_{1}^h, z_{2}^h)
            \\
            z_{3}^h
            }
        ,
        \phi^h
    \Big\rangle _{X^*, X}
    +
    \big\langle
        \bbf(z^h) u
        ,
        \phi^h
    \big\rangle _{X^*,X}\end{equation}
for all test functions $\phi^h \in X_1^h \times X_2^h \times X_3^h$. Note that we have dropped the superscript~$\tau$ to indicate continuity in time.
As in \Cref{prop:energy-balance-discrete}, solutions of~\eqref{eq:time-continuous-spatially-discrete} satisfy a dissipation inequality.
In fact, we can replace the times $t_{j-1}$ and $t_{j}$ in~\eqref{eq:dissipation-discrete} by arbitrary time points $0\le s < t \le T$.
Although spatially discrete and reduced order models share many similarities, in what follows we discuss them separately.

\subsection{Space discretization}\label{subsec:space}
As dimensions for the spatially discrete models, we introduce
\[
    n_i^h \coloneqq \dim(X_i^h)\in \N, \
    i=1,2,3
    \qquad\text{and}\qquad
    n^h \coloneqq n_1^h + n_2^h + n_3^h.
\]

As an alternative to~\eqref{eq:time-continuous-spatially-discrete}, we can state the dynamics for the corresponding coefficient vectors $c_i^h \in \R^{n_i^h}$ of $z_i^h$, $i=1,2,3$.
For this, we fix bases of $X_i^h$, $i=1,2,3$, and set $c^h \coloneqq (c_1^h, c_2^h, c_3^h) \in \R^{n^h}$.
We denote the canonical interpolation operator (or any suitable generalization of it) mapping a coefficient vector to the corresponding element in $X_i^h$ by
\[
    \interp_i^h \colon \R^{n_{i,h}} \to X_i^h,\qquad
    i=1,2,3
\]

and set $\interp_{12}^h\colon \R^{n_1^h} \times \R^{n_2^h} \to X_1^h \times X_2^h,~ (c_1^h, c_2^h) \mapsto \big(\interp_1^h c_1^h, \interp_2^h c_2^h\big)$.
Similarly, we define $\interp^h \colon \R^{n^h} \to X_1^h \times X_2^h \times X_3^h$ by
\[
    (c_1^h, c_2^h, c_3^h)
    \mapsto \big(\interp_1^h c_1^h, \interp_2^h c_2^h, \interp_3^h c_3^h \big).
\]
Conversely, the functions $\coeff_i^h \colon X_i^h \to \R^{n_i^h}$ map elements of $X_i^h$ to their coefficient vectors with respect to the chosen basis.
We can now define the space-discrete energy $\ham^h\colon \R^{n_1^h} \times \R^{n_2^h} \to \R$ by
\begin{equation*}    \ham^h(c_1^h, c_2^h)
    \coloneqq \ham\big(\interp_1^h c_1^h, \interp_2^h c_2^h\big).\end{equation*}
Since the interpolation operators $\interp_1^h$ and $\interp_2^h$ are linear, using~\eqref{eq:spatial-projection-property} for $c_{12}^h = (c_1^h, c_2^h) \in \R^{n_1^h}\times\R^{n_2^h}$ with $\patz i \ham(\interp_{12}^h c_{12}^h) \in X_i$ for $i=1,2$ and arbitrary $d_{12}^h = (d_1^h, d_2^h) \in \R^{n_1^h}\times\R^{n_2^h}$ we obtain

\begin{equation*}
\begin{aligned}
    \nabla \ham^h(c_{12}^h)\tp d_{12}^h
    &
    =
    \frac{\mathrm{d}}{\mathrm{d}s} \ham^h(c_{12}^h + s d_{12}^h)\Big|_{s=0}
    \\
    &
    =
    \frac{\mathrm{d}}{\mathrm{d}s} \ham(\interp_{12}^h c_{12}^h + s\, \interp_{12}^h d_{12}^h)\Big|_{s=0}
    \\
    &
    =
    \Big\langle
        \patz 1 \ham(\interp_{12}^h c_{12}^h)
        ,
        \interp_{1}^h d_{1}^h
    \Big\rangle _{Z_1}
    +
    \Big\langle
        \patz 2 \ham(\interp_{12}^h c_{12}^h)
        ,
        \interp_{2}^h d_{2}^h
    \Big\rangle _{Z_2}
    \\
    &
    =
    \Big\langle
        \proj_1^h\patz 1 \ham(\interp_{12}^h c_{12}^h)
        ,
        \interp_{1}^h d_{1}^h
    \Big\rangle _{Z_1}
    +
    \Big\langle
        \proj_2^h\patz 2 \ham(\interp_{12}^h c_{12}^h)
        ,
        \interp_{2}^h d_{2}^h
    \Big\rangle _{Z_2}
    \\
    &
    =
    \Big(\Mmat_1^h \zeta_1^h\big(\proj_1^h\patz 1 \ham(\interp_{12}^h c_{12}^h)\big)\Big)\tp d_1^h
    +
    \Big(\Mmat_2^h \zeta_2^h\big(\proj_2^h\patz 2 \ham(\interp_{12}^h c_{12}^h)\big)\Big)\tp d_2^h
    \\
    &
    =
    \begin{bmatrix}        \Mmat_1^h \zeta_1^h\big(\proj_1^h\patz 1 \ham(\interp_{12}^h c_{12}^h)\big)
        \\
        \Mmat_2^h \zeta_2^h\big(\proj_2^h\patz 2 \ham(\interp_{12}^h c_{12}^h)\big)\end{bmatrix}\tp d_{12}^h,
\end{aligned}\end{equation*}

where $\Mmat_1^h$ and $\Mmat_2^h$ are the mass matrices corresponding to the bases of $X_1$ and $X_2$, respectively.
Hence, we have

\begin{equation*}    \nabla \ham^h(c_1^h, c_2^h)
    =
    \begin{bmatrix}        \Mmat_1^h \zeta_1^h\big(\proj_1^h\patz 1 \ham(\interp_{1}^h c_{1}^h,\interp_{2}^h c_{2}^h)\big)
        \\
        \Mmat_2^h \zeta_2^h\big(\proj_2^h\patz 2 \ham(\interp_{1}^h c_{1}^h,\interp_{2}^h c_{2}^h)\big)\end{bmatrix}.\end{equation*}

Accordingly, we set $\patc i \ham^h(c_{12}^h) \coloneqq \Mmat_i^h \zeta_i^h(\proj_i^h\patz i \ham(\interp_{12}^h c_{12}^h))$, $i=1,2$.
Now, define $\bar{\jbf}^h, \bar{\rbf}^h \colon \R^{n^h} \to \R^{n^h}$ and~$\bar{\bbf}^h \colon \RR^{n^h} \to \boundedlinear(U, \R^{n^h})$ via
\begin{gather*}    \bar{\jbf}^h v^h
    \coloneqq
    \vec{
        \big\langle \jbf (\interp^h v^h), \interp^h e_j \big\rangle _{X^*, X}
    }_{j=1,\dots,n^h}
    ,\quad
    \bar{\rbf}^h v^h
    \coloneqq
    \vec{
        \big\langle \rbf (\interp^h v^h), \interp^h e_j \big\rangle _{X^*, X}
    }_{j=1,\dots,n^h}
    ,
    \\
    \bar{\bbf}^h (c^h) u
    \coloneqq
    \vec{
        \big\langle \bbf (\interp^h c^h) u, \interp^h e_j \big\rangle _{X^*, X}
    }_{j=1,\dots,n^h}\end{gather*}
for all $v^h \in \R^{n^h}$ and $u\in U$, where $e_j \in \R^{n^h}$ denotes the $j$-th unit vector.
Note that~\eqref{eq:J-R-properties} implies $(v^h)\tp \bar{\jbf}^h v^h = 0$ and $(v^h)\tp \bar{\rbf}^h v^h \geq 0$ for all $v^h \in \R^{n^h}$.
\begin{remark}\label{rem:space-discrete-linear-operators}
    If $\jbf$ and $\rbf$ are linear, then we can identify $\bar{\jbf}^h$ and $\bar{\rbf}^h$ with matrices. The same holds true for $\bar{\bbf}^h$ if $\bbf$ is linear and $U$ finite-dimensional.
\end{remark}
An equivalent formulation of~\eqref{eq:time-continuous-spatially-discrete} is given by

\begin{equation}\label{eq:time-continuous-spatially-discrete-coefficients-with-mass-matrices}
    \vec{
        \patc 1 \ham^h(c_1^h, c_2^h)
        \\
        \Mmat_2^h \pat c_2^h
        \\
        0
    }
    =
    \big(\bar{\jbf}^h - \bar{\rbf}^h\big)
    \vec{
        \pat c_1^h
        \\
        (\Mmat_2^h)^{-1} \patc 2 \ham^h(c_1^h, c_2^h)
        \\
        c_3^h
    }
    +
    \bar{\bbf}^h{(c^h)} u.\end{equation}

To cast~\eqref{eq:time-continuous-spatially-discrete-coefficients-with-mass-matrices} into the form~\eqref{eq:model-strong}, define $\Kmat^h \coloneqq \diag(\eye_{n_1^h}, (\Mmat_2^h)^{-1}, \eye_{n_3^h}) \in \R^{n^h, n^h}$ as well as $\jbf^h, \rbf^h \colon \R^{n^h} \to \R^{n^h}$ and~$\bbf^h \colon \R^{n^h} \to \boundedlinear(U, \R^{n^h})$ by
\begin{equation*}    \jbf^h(v^h) \coloneqq \Kmat^h \bar{\jbf}^h (\Kmat^h v^h),
    \qquad
    \rbf^h(v^h) \coloneqq \Kmat^h \bar{\rbf}^h (\Kmat^h v^h),
    \qquad
    \bbf^h{(c^h)} \coloneqq \Kmat^h \bar{\bbf}^h{(c^h)}\end{equation*}
for all $v^h \in \R^{n^h}$ such that~\eqref{eq:time-continuous-spatially-discrete-coefficients-with-mass-matrices} can be written as

\begin{equation}\label{eq:time-continuous-spatially-discrete-coefficients}
    \begin{aligned}
        \vec{
        \patc 1 \ham^h(c_1^h, c_2^h)
        \\
        \dot{c}_2^h
        \\
        0
    }
    & =
    \big(\jbf^h - \rbf^h\big)
    \vec{
        \dot{c}_1^h
        \\
        \patc 2 \ham^h(c_1^h, c_2^h)
        \\
        c_3^h
    }
    +
    \bbf^h(c^h) u.
    \end{aligned}\end{equation}

Due to~\eqref{eq:J-R-properties} and the definitions of $\bar{\jbf}^h, \bar{\rbf}^h$, and $\bar{\bbf}^h$, we have $(v^h)\tp \jbf^h (v^h) = 0$ and $(v^h)\tp \rbf^h (v^h) \geq 0$ for all $v^h \in \R^{n^h}$ such that~\eqref{eq:time-continuous-spatially-discrete-coefficients}, together with an appropriately defined output~$y^h$, is indeed of the form~\eqref{eq:model-strong}.
In particular, for solutions $c^h = (c_1^h, c_2^h, c_3^h)$ of~\eqref{eq:time-continuous-spatially-discrete-coefficients} we obtain the dissipation inequality
\begin{equation*}
        \ham^h(c_{1}^h(\theta_2), c_{2}^h(\theta_2)) - \ham^h(c_{1}^h(\theta_1), c_{2}^h(\theta_1))
        \leq
        \int_{\theta_1}^{\theta_2}
            \langle
                y^h
                ,
                u
            \rangle _{U^*, U}
        \dt\end{equation*}
for all $0 \leq \theta_1 \leq \theta_2$ and $u \colon [0,T] \to U$.

To conclude, switching to the coefficient representation of~$z^h$ eliminates the need for the projection~$\proj_2^h$ if the discrete energy functional~$\ham^h$ is defined in a suitable manner.
Moreover, the coefficient dynamics can be formulated in the form~\eqref{eq:model-strong} and, hence, satisfy an energy balance and the corresponding dissipation inequality.

In the special case where $z_2 = \bullet$, the presented structure-preserving space discretization scheme is simply a Galerkin projection.
This has been exploited, e.g., in \cite{EggHS21}, where a formulation with $z_1$ only has been considered.
It turns out that this pure $z_1$-formulation is not only amenable for structure-preserving schemes when approximating the infinite-dimensional state in a finite-dimensional subspace, but also when using a finite-dimensional nonlinear manifold instead.
This is formalized in the following proposition, where the discretized system is obtained by the Dirac--Frenkel variational principle, cf.~\cite{Dir30,Fre34,Lub08}, which enforces orthogonality of the residual and the tangent space of the approximation manifold.
In particular, the proposition applies the Neural Galerkin method \cite{BruPV24}, which is based on approximating the state by an artificial neural network with time-dependent weights.

\begin{proposition}
	\label{prop:Dirac-Frenkel}
	Consider the infinite-dimensional system
	\begin{subequations}\label{eq:model-z1-weak}
	    \begin{align}	    	\label{eq:model-z1-weak-state}
		    \Blangle
		    \patz{} \ham(z)
		    ,
		    \phi
		    \Brangle_{X^*, X}
		    & =
		    \Blangle
		        (\jbf - \rbf)\pat z
		        ,
		        \phi
		    \Brangle_{X^*, X}
		    +
		    \blangle
		        \bbf(z) u
		        ,
		        \phi
		    \brangle_{X^*,X}
		    ,
		    \\
		    y
		    \label{eq:model-z1-weak-output}
		    & =
		    \bbf(z)^* \pat z\end{align}\end{subequations}
	as special case of \eqref{eq:model-weak} with $z_2 = z_3 = \bullet$ and $X$ being a Hilbert space.
	Furthermore, consider an approximation ansatz $z\approx \interp^h(c^h)$ with continuously differentiable mapping $\interp^h\colon \R^{n^h}\to X^h$.
	Then, the corresponding discretized system resulting from the Dirac--Frenkel variational principle is of the form
	\begin{subequations}\label{eq:model-z1-Dirac-Frenkel}
		\begin{align}			\label{eq:model-z1-Dirac-Frenkel-state}
		    \partial_{c^h} \ham^h(c^h)
		    & = \jbf^h(c^h,\dot{c}^h) - \rbf^h(c^h,\dot{c}^h)
		    +
		    \bbf^h(c^h) u,
		    \\
		    \label{eq:model-z1-Dirac-Frenkel-output}
		    y^h
		    &
		    =
		    \bbf^h(c^h)^* \dot{c}^h,\end{align}\end{subequations}
	with $\ham^h\vcentcolon= \ham\circ\interp^h$, $\bbf^h\colon \R^{n^h} \to \boundedlinear(U, \R^{n^h})$, and $\jbf^h, \rbf^h \colon \R^{n^h}\times\R^{n^h} \to \R^{n^h}$ satisfying
	\begin{equation*}		(w^h)\tp \jbf^h(v^h,w^h) = 0\quad\text{and}\quad (w^h)\tp \rbf^h (v^h,w^h) \geq 0\quad\text{for all }v^h,w^h \in \R^{n^h}.\end{equation*}
\end{proposition}

\begin{proof}
	The Dirac--Frenkel variational principle enforces the residual to be orthogonal to the tangent space of the manifold parametrized by $\interp^h$, cf.~\cite[Sec.~II.1]{Lub08}.
	This leads to a discretized state equation of the form
	\begin{equation*}		\Blangle
		    \patz{} \ham(\interp^h(c^h))
		    ,
		    \phi
		    \Brangle_{X^*, X}
		     =
		    \Blangle
		        (\jbf - \rbf) (D\interp^h(c^h)\dot{c}^h)
		        ,
		        \phi
		    \Brangle_{X^*, X}
		    +
		    \blangle
		        \bbf(\interp^h(c^h)) u
		        ,
		        \phi
		    \brangle_{X^*,X}\end{equation*}
	for all $\phi$ in the tangent space of the manifold parametrized by $\interp^h$.
	Here, we can replace $\phi$ by the partial derivatives $\derivative\interp^h(c^h)e_i$ for $i=1,\ldots,n^h$.
	Then, straightforward calculations lead to the discretized state equation \eqref{eq:model-z1-Dirac-Frenkel-state} with
	\begin{align*}	    \jbf^h(v^h,w^h)
	    &\coloneqq
	    \vec{
	        \big\langle \jbf (\derivative\interp^h(v^h)w^h), \derivative\interp^h(v^h) e_j \big\rangle _{X^*, X}
	    }_{j=1,\dots,n^h}
	    ,\\
	    \rbf^h(v^h,w^h)
	    &\coloneqq
	    \vec{
	        \big\langle \rbf (\derivative\interp^h(v^h)w^h), \derivative\interp^h(v^h) e_j \big\rangle _{X^*, X}
	    }_{j=1,\dots,n^h}
	    ,
	    \\
	    \bbf^h{(v^h)} u
	    &\coloneqq
	    \vec{
	        \big\langle \bbf {(\interp^h(v^h))} u, \derivative\interp^h(v^h) e_j \big\rangle _{X^*, X}
	    }_{j=1,\dots,n^h}.\end{align*}
	The properties of $\jbf$ and $\rbf$ imply $(w^h)\tp \jbf^h(v^h,w^h)=0$ and $(w^h)\tp \rbf^h(v^h,w^h)\ge 0$ for all $v^h,w^h \in \R^{n^h}$.
	Similarly, the output equation \eqref{eq:model-z1-Dirac-Frenkel-output} is obtained by substituting the approximation ansatz $z\approx \interp^h(c^h)$ into the output equation \eqref{eq:model-z1-weak-output} of the original system.
\end{proof}

It should be emphasized that the discretized system \eqref{eq:model-z1-Dirac-Frenkel} in \Cref{prop:Dirac-Frenkel} does not have the same structure as the original system, since $\jbf^h$ and $\rbf^h$ explicitly depend on the state~$c^h$.
However, the resulting structure with explicitly state-dependent $\jbf^h$ and $\rbf^h$ still guarantees a power balance as in \Cref{prop:energy-balance}, which can be shown by the same calculations. 
If system~\eqref{eq:model-z1-Dirac-Frenkel} is discretized in time using piecewise polynomials, one can compute~$c^h$ from the derivative~$\dot{c}^h$ by integrating and using the initial value.
Hence, in that case, the system can be realized without explicitly depending on~$c^h$.

\subsection{Model reduction}\label{subsec:model-reduction}
As mentioned in the introduction, our proposed model reduction scheme is closely related to~\cite{ChaBG16}.
For model reduction, we assume that the state and input spaces are finite-dimensional, i.e., we have~$X_i = Z_i = \R^{n_i}$ and~$U = \RR^m$ for some~$n_i, m \in \N$,~$i=1,2,3$.
Furthermore, we assume that
\[
    X_i^h
    = \R^{\nr_i}
    = \Vmat_i\tp X_i
\]
with an orthogonal matrix $\Vmat_i \in \R^{n_i,\nr_i}$, $\nr_i \leq n_i$, $i=1,2,3$.

Note, however, that also non-orthogonal matrices $\Vmat_i$ with full column rank can be considered.

To emphasize that we are working with reduced order models, we denote the reduced order states by $\zr_i$ instead of~$z_i^h$.
Since $X_i^h = \Vmat_i\tp X_i$, we essentially approximate $z_i \approx \Vmat_i \zr_i$, and $v\mapsto \Vmat_i v$ takes the role of the interpolation operator $\interp_i^h$.
Furthermore, since $\Vmat_i$ is orthogonal, the projection $\projmor_i^h$ is given by
\[
    \projmor_i^h \colon \R^{n_i} \to \R^{n_i},\quad
    z \mapsto \Vmat_i \Vmat_i\tp z.
\]
Similar to the space discretization, as an alternative to~\eqref{eq:time-continuous-spatially-discrete} we state the dynamics for the reduced order states $\zr_i$.
Accordingly, we define the reduced order energy $\tilde{\ham}\colon \R^{\nr_1} \times \R^{\nr_2} \to \R$ by

\begin{equation*}    \hamr(\zr_1, \zr_2) \coloneqq \ham(\Vmat_1 \zr_1, \Vmat_2 \zr_2)\end{equation*}

for all $(\zr_1, \zr_2) \in \R^{\nr_1} \times \R^{\nr_2}$.
Similar to the calculation in the previous subsection, for any $\zr_{12}, \vr_{12} \in X_1^h \times X_2^h$ with $\zr_{12} = (\zr_1, \zr_2)$ and $\vr_{12}=(\vr_1, \vr_2)$ with the chain rule we obtain

\[
    \nabla \hamr(\zr_{12})
    = \vec{\Vmat_1\tp \patz 1 \ham(\Vmat_{12} \zr_{12}) \\ \Vmat_2\tp \patz 2 \ham(\Vmat_{12} \zr_{12})}
    = \vec{\Vmat_1\tp \projmor_1^h \patz 1 \ham(\Vmat_{12} \zr_{12}) \\ \Vmat_2\tp \projmor_2^h \patz 2 \ham(\Vmat_{12} \zr_{12})},
\]
where $\Vmat_{12} \coloneqq \diag(\Vmat_1, \Vmat_2)$, and we can set $\patzr i \hamr(\zr_1, \zr_2) \coloneqq \Vmat_i\tp \patz i \ham(\Vmat_1 \zr_1, \Vmat_2 \zr_2)$ as before.
Now, the reduced order dynamics~\eqref{eq:time-continuous-spatially-discrete} can be written as
\begin{equation}	\label{eq:ROM_after_projection}
    \vec{
        \Vmat_1\tp\patz 1 \ham(\Vmat_1 \zr_1, \Vmat_2 \zr_2)
        \\
        \Vmat_2\tp \Vmat_2 \pat \zr_2
        \\
        0
    }
    = \Vmat\tp
    (\jbf - \rbf)
    \vec{
        \Vmat_1 \pat \zr_1
        \\
        \Vmat_2 \Vmat_2\tp \patz 2 \ham(\Vmat_1 \zr_1, \Vmat_2 \zr_2)
        \\
        \Vmat_3 \zr_3
    }
    +
    \Vmat\tp\bbf(\zr) u\end{equation}
or, equivalently,
\begin{equation}\label{eq:time-continuous-reduced-order}
    \vec{
        \patzr 1 \hamr(\zr_1, \zr_2)
        \\
        \pat \zr_2
        \\
        0
    }
    =
    (\jbfr - \rbfr)
    \vec{
        \pat \zr_1
        \\
        \patzr 2 \hamr(\zr_1, \zr_2)
        \\
        \zr_3
    }
    +
    \bbfr{(\zr)} u\end{equation}
with
\begin{equation*}    \jbfr(v) \coloneqq \Vmat\tp \jbf (\Vmat v)
    ,\qquad
    \rbfr(v) \coloneqq \Vmat\tp \rbf (\Vmat v)
    ,\qquad
    \bbfr(\zr) \coloneqq \Vmat\tp \bbf (\Vmat \zr)
    ,\qquad
    \Vmat \coloneqq \diag(\Vmat_1, \Vmat_2, \Vmat_3).\end{equation*}
Since $v\tp \jbfr (v) = 0$ and $v\tp \rbfr (v) \geq 0$ for all $v\in\R^{\nr}$, model~\eqref{eq:time-continuous-reduced-order} is of the form~\eqref{eq:model-strong}.
Consequently, for sufficiently smooth solutions $\zr = (\zr_1, \zr_2, \zr_3)$ of~\eqref{eq:time-continuous-reduced-order} and $u$ as before, we obtain the dissipation inequality
\begin{equation}\label{eq:dissipation-reduced-order}
\begin{aligned}
    \hamr((\zr_1, \zr_2)(\theta_2)) - \hamr((\zr_1, \zr_2)(\theta_1))
    &
    =
    \int_{\theta_1}^{\theta_2}
    -
    \Bigtextvec{
        \pat \zr_1
        \\
        \patzr 2 \hamr(\zr_1, \zr_2)
        \\
        \zr_3
    }\tp
    \rbfr\,
    \Bigtextvec{
        \pat \zr_1
        \\
        \patzr 2 \hamr(\zr_1, \zr_2)
        \\
        \zr_3
    }
    +
    \yr\tp
    u
    \dt
    \\
    &
    \leq
    \int_{\theta_1}^{\theta_2}
    \yr\tp
    u
    \dt
\end{aligned}\end{equation}
for all $0 \leq \theta_1 \leq \theta_2$.

\begin{remark}
In general, the evaluation of $\jbfr$, $\rbfr$, $\bbfr$, $\patzr 1 \hamr$, and $\patzr 2 \hamr$ involves the corresponding full order operator and, hence, scales with the dimension of the full order model.
To address this issue, an additional approximation of the reduced order model, known as hyperreduction, is usually performed to obtain a computationally cheap surrogate model.
While structure-preserving hyperreduction is not within the scope of this paper, we refer to~\cite{ChaBG16} for a corresponding approach for port-Hamiltonian systems{, i.e., for the special case with $z_1 = z_3 = \bullet$ and with linear operators $\jbf$ and $\rbf$}.
\end{remark}

\begin{remark}
	Similar to the spatial discretization in \Cref{subsec:space}, the model reduction approach presented above is based on a Galerkin projection, i.e., the trial and test spaces coincide.
	In contrast, the method presented in \cite{ChaBG16}, which considers the special case with empty $z_1$ and $z_3$, is based on a Petrov--Galerkin projection.
	In our notation, this corresponds to replacing $\Vmat_2\tp$ in the second block equation of \eqref{eq:ROM_after_projection} by some matrix $\Wmat_2\tp$ satisfying the biorthogonality condition $\Wmat_2\tp\Vmat_2=\Imat$.
	In this case, the term $\Vmat_2 \Vmat_2\tp \patz 2 \ham(\Vmat_1 \zr_1, \Vmat_2 \zr_2)$ in \eqref{eq:ROM_after_projection} needs to be replaced by $\Wmat_2 \Vmat_2\tp \patz 2 \ham(\Vmat_1 \zr_1, \Vmat_2 \zr_2)$ to preserve the structure.
	Using a Petrov-Galerkin projection provides some additional degree of freedom which may be exploited, e.g., by choosing $\Wmat_2$ such that the error in the additional approximation $\patz 2 \ham(\Vmat_1 \zr_1, \Vmat_2 \zr_2)\approx \Wmat_2 \Vmat_2\tp\patz 2 \ham(\Vmat_1 \zr_1, \Vmat_2 \zr_2)$ is as small as possible.
	For instance, in \cite{ChaBG16} the authors propose to determine $\Wmat_2$ via a proper orthogonal decomposition applied to a snapshot matrix of the gradient of the Hamiltonian.
\end{remark}

\begin{remark}
	Recently, in \cite{GlaM26} the authors present a new structure-preserving model reduction scheme for port-Hamiltonian systems based on a Petrov--Galerkin projection.
	In contrast to the approach presented here, their method does not require the additional approximation step $\nabla \ham(\Vmat \zr)\approx \Vmat \Vmat\tp\patz 2 \ham(\Vmat\zr)$.
	A key assumption in their approach is that $\Jmat-\Rmat$ is invertible.
	The here considered $z_1$--$z_2$--$z_3$-formulation yields an alternative interpretation of their method.
	When ignoring the output equation, the approach in \cite{GlaM26} is equivalent to the following procedure:
	\begin{enumerate}
		\item Transform the $z_2$-formulation to an equivalent $z_1$-formulation by multiplying the state equation by $(\Jmat-\Rmat)^{-1}$ from the left. Here, one exploits that the inverse of a matrix with negative semi-definite symmetric part has again a negative semi-definite symmetric part, cf.~\cite[Lem.~20]{AchAM23}.
		\item Perform a structure-preserving Galerkin projection.
		\item Assuming that $\Vmat\tp(\Jmat-\Rmat)^{-1}\Vmat$ is invertible, transform back to an equivalent $\zr_2$-formulation by multiplying by $(\Vmat\tp(\Jmat-\Rmat)^{-1}\Vmat)^{-1}$ from the left.
	\end{enumerate}
	The Galerkin projection onto a linear subspace in the second step can be replaced by a structure-preserving projection onto a nonlinear submanifold, cf.~\Cref{prop:Dirac-Frenkel}, which is also (implicitly) exploited in \cite{GlaM26}.
	Moreover, the alternative interpretation from above can be extended to the case when the output equation is taken into account, but then another condition for the column space of $\bbf$ is needed to arrive at a port-Hamiltonian system in the end, cf.~\cite[Thm.~3.1]{GlaM26}.
\end{remark}

\begin{remark}
	In this paper, we focus on the preservation of the energy-based structure, whereas a detailed treatment of the preservation of algebraic constraints is not within the scope of this manuscript.
	For model reduction approaches for differential--algebraic equations, which explicitly take into account the preservation of algebraic constraints, we refer to \cite{Ebe10,MoeRS11,BeaGM22} and the survey paper \cite{BenS17}.
\end{remark}

In \cite[Thm.~2.2]{ChaBG16}, the authors also present an error bound for the reduced-order model (ROM) resulting from the Petrov--Galerkin projection in combination with the projection of the gradient of the Hamiltonian.
This error bound directly applies to our setting in the special case where $z_1$ and $z_3$ are empty and $\jbf$ and $\rbf$ are matrices as in \eqref{eq:model:altschulz-finite-dim}.
While an extension of this result to the general structures \eqref{eq:model:altschulz-finite-dim} and \eqref{eq:model-strong} is not within the scope of this paper, in the following we provide an error bound for the special case of a semi-explicit full-order model (FOM) with index one.
In particular, we consider the special case of \eqref{eq:model:altschulz-finite-dim} with $\Jmat_{11}=0$, $\Rmat_{11}=\eye$, $\Jmat_{21}=\Rmat_{21}$, $\Jmat_{31}=\Rmat_{31}$, $\det(\Jmat_{33}-\Rmat_{33})\ne 0$ and without in- and outputs, where $\Jmat_{ij}$ and $\Rmat_{ij}$ denote the $(i,j)$-blocks of the $3\times3$ block matrices $\Jmat$ and $\Rmat$, respectively.
The resulting system reads
\begin{equation*}	\vec{
		\patz 1 \ham(z_1, z_2)\\
		\dot{z}_2\\
		0
	}
	=
	\begin{bmatrix}		-\eye 	& \Jmat_{12}-\Rmat_{12} & \Jmat_{13}-\Rmat_{13} \\
		0 	& \Jmat_{22}-\Rmat_{22} & \Jmat_{23}-\Rmat_{23} \\
		0 	& \Jmat_{32}-\Rmat_{32} & \Jmat_{33}-\Rmat_{33}\end{bmatrix}
	\vec{
		\dot{z}_1\\
		\patz 2 \ham(z_1, z_2)\\
		z_3
	}\end{equation*}
or, equivalently,
\begin{equation}	\label{eq:semi-explicit-FOM}
	\vec{
		\dot{z}_1\\
		\dot{z}_2\\
		0
	}
	=
	\begin{bmatrix}		-\eye 	& \Jmat_{12}-\Rmat_{12} & \Jmat_{13}-\Rmat_{13} \\
		0 	& \Jmat_{22}-\Rmat_{22} & \Jmat_{23}-\Rmat_{23} \\
		0 	& \Jmat_{32}-\Rmat_{32} & \Jmat_{33}-\Rmat_{33}\end{bmatrix}
	\vec{
		\patz 1 \ham(z_1, z_2)\\
		\patz 2 \ham(z_1, z_2)\\
		z_3
	}.\end{equation}
Here, the differential equations and algebraic constraints are clearly separated, which allows us to only reduce the dimension of the differential part and keep the dimension of the algebraic part by setting $\Vmat_3=\eye$, as is often done in model reduction for differential--algebraic equations, see e.g.~\cite{Ebe10,MoeRS11}.
The resulting ROM reads
\begin{equation}	\label{eq:semi-explicit-ROM}
	\vec{
		\dot{\zr}_1\\
		\dot{\zr}_2\\
		0
	}
	=
	\begin{bmatrix}		-\eye 	& \Vmat_1\tp(\Jmat_{12}-\Rmat_{12})\Vmat_2 & \Vmat_1\tp(\Jmat_{13}-\Rmat_{13}) \\
		0 	& \Vmat_2\tp(\Jmat_{22}-\Rmat_{22})\Vmat_2 & \Vmat_2\tp(\Jmat_{23}-\Rmat_{23}) \\
		0 	& (\Jmat_{32}-\Rmat_{32})\Vmat_2 & \Jmat_{33}-\Rmat_{33}\end{bmatrix}
	\vec{
		\patzr 1 \hamr(\zr_1, \zr_2)\\
		\patzr 2 \hamr(\zr_1, \zr_2)\\
		\zr_3
	}.\end{equation}
The following proposition yields an error bound between the states of the semi-explicit FOM \eqref{eq:semi-explicit-FOM} and the corresponding ROM \eqref{eq:semi-explicit-ROM}.

\begin{proposition}
	\label{thm:error_bound}
	Let $(z_1,z_2,z_3)\colon [0,T]\to\R^{n_1}\times \R^{n_2}\times\R^{n_3}$ be a solution of the semi-explicit FOM \eqref{eq:semi-explicit-FOM} with initial value $(z_1(0),z_2(0),z_3(0))\in\R^{n_1}\times \R^{n_2}\times\R^{n_3}$ and $(\zr_1,\zr_2,\zr_3)\colon [0,T]\to\R^{\nr_1}\times \R^{\nr_2}\times\R^{\nr_3}$ be a solution of the ROM \eqref{eq:semi-explicit-ROM} with initial value $(\Vmat_1\tp z_1(0),\Vmat_2\tp z_2(0),z_3(0))$.
	Moreover, let $\nabla\ham$ be Lipschitz continuous with Lipschitz constant $K$.
	Then, the approximation error satisfies the bound
	\begin{align*}		\Big\lVert\vec{
		z_1-\Vmat_1\zr_1\\
		z_2-\Vmat_2\zr_2
		}\Big\rVert_{L^2(0,T;\R^{n_1+n_2})}^2
		&\le (1+K\norm{\Pmat}Te^{C_1T})\Big\lVert\vec{
		(\Imat-\projmor_1^h)z_1\\
		(\Imat-\projmor_2^h)z_2
		}\Big\rVert_{L^2(0,T;\R^{n_1+n_2})}^2\\
		&\ \ +\norm{\Pmat}Te^{C_1T}\lVert (\Imat-\projmor_2^h)\patz 2 \ham(z_1, z_2)\rVert_{L^2(0,T;\R^{n_1+n_2})}^2,\\[0.2cm]
		\lVert z_3-\zr_3\rVert_{L^2(0,T;\R^{n_3})}^2
		&\le KC_2(1+K\norm{\Pmat}Te^{C_1T})\Big\lVert\vec{
		(\Imat-\projmor_1^h)z_1\\
		(\Imat-\projmor_2^h)z_2
		}\Big\rVert_{L^2(0,T;\R^{n_1+n_2})}^2\\
		&\ \ +C_2(1+K\norm{\Pmat}Te^{C_1T})\lVert (\Imat-\projmor_2^h)\patz 2 \ham(z_1, z_2)\rVert_{L^2(0,T;\R^{n_1+n_2})}^2\end{align*}
	with $C_1 \vcentcolon= \norm{\Pmat}(1+3K)$, $C_2\vcentcolon= \norm{(\Jmat_{33}-\Rmat_{33})^{-1}(\Jmat_{32}-\Rmat_{32})}$, and
	\begin{equation*}		\Pmat \vcentcolon=
		\begin{bmatrix}			-\Imat & \Jmat_{12}-\Rmat_{12}+(\Jmat_{13}-\Rmat_{13})(\Jmat_{33}-\Rmat_{33})^{-1}(\Jmat_{32}-\Rmat_{32})\\
			0 & \Jmat_{22}-\Rmat_{22}+(\Jmat_{23}-\Rmat_{23})(\Jmat_{33}-\Rmat_{33})^{-1}(\Jmat_{32}-\Rmat_{32})\end{bmatrix}
		.\end{equation*}
\end{proposition}

\begin{proof}
	The proof combines ideas from \cite{Ebe10} and the proof of \cite[Thm.~2.2]{ChaBG16}.
	We first split the error in the differential variables $\differror\vcentcolon=\begin{bsmallmatrix}
		z_1-\Vmat_1\zr_1\\
		z_2-\Vmat_2\zr_2
	\end{bsmallmatrix}$ into two components via
	\begin{equation*}		\differror = \underbrace{\vec{
		z_1-\projmor_1^hz_1\\
		z_2-\projmor_2^hz_2
		}}_{=\vcentcolon\rho}+\underbrace{\vec{
		\projmor_1^hz_1-\Vmat_1\zr_1\\
		\projmor_2^hz_2-\Vmat_2\zr_2
		}}_{=\vcentcolon\theta}\end{equation*}
	and introduce
	\begin{equation*}		\thetar\vcentcolon= \vec{
		\Vmat_1\tp z_1-\zr_1\\
		\Vmat_2\tp z_2-\zr_2
		}.\end{equation*}
	We have $\theta=\Vmat_{12}\thetar$ and the special choice of the initial values for the ROM state implies $\thetar(0)=0$.
	For any $t\in[0,T]$, we obtain
	\begin{equation}		\label{eq:product_rule_for_squared_norm}
		\frac{\mathrm{d}}{\dt}(\norm{\thetar(t)}^2) = 2\thetar(t)\tp \dot{\thetar}(t) = 2\thetar(t)\tp\vec{
		\Vmat_1\tp \dot{z}_1(t)-\dot{\zr}_1(t)\\
		\Vmat_2\tp \dot{z}_2(t)-\dot{\zr}_2(t)
		}.\end{equation}
	For the sake of brevity, we omit the time dependency in the following, and use the dynamics of the FOM \eqref{eq:semi-explicit-FOM} and the ROM \eqref{eq:semi-explicit-ROM} to get
	\begin{equation}		\label{eq:thetar_dot}
		\begin{aligned}
			\Vmat_1\tp \dot{z}_1-\dot{\zr}_1 = &\Vmat_1\tp(\patz 1 \ham(\Vmat_{12}\zr_{12})-\patz 1 \ham(z_{12}))+\Vmat_1\tp(\Jmat_{13}-\Rmat_{13})(z_3-\zr_3)\\
			&+\Vmat_1\tp(\Jmat_{12}-\Rmat_{12})(\patz 2\ham(z_{12})-\projmor_2^h\patz2\ham(\Vmat_{12}\zr_{12})),\\[0.2cm]
			\Vmat_2\tp \dot{z}_2-\dot{\zr}_2 = &\Vmat_2\tp(\Jmat_{22}-\Rmat_{22})(\patz 2\ham(z_{12})-\projmor_2^h\patz2\ham(\Vmat_{12}\zr_{12}))+\Vmat_2\tp(\Jmat_{23}-\Rmat_{23})(z_3-\zr_3)
		\end{aligned}\end{equation}
	with $z_{12}\vcentcolon=(z_1, z_2)$.
	Solving the respective last block equations in the FOM \eqref{eq:semi-explicit-FOM} and the ROM \eqref{eq:semi-explicit-ROM} for $z_3$ and $\zr_3$, we obtain
	\begin{equation}		\label{eq:z3_in_terms_of_z1_and_z2}
		z_3-\zr_3 = (\Jmat_{33}-\Rmat_{33})^{-1}(\Jmat_{32}-\Rmat_{32})(\patz 2\ham(z_{12})-\projmor_2^h\patz2\ham(\Vmat_{12}\zr_{12})).\end{equation}
	Combining \eqref{eq:thetar_dot} and \eqref{eq:z3_in_terms_of_z1_and_z2} yields
	\begin{equation*}		\dot{\thetar} = \Vmat_{12}\tp\Pmat
		\begin{bmatrix}			\Imat & 0\\
			0 & \projmor_2^h\end{bmatrix}
		(\nabla\ham(z_{12})-\nabla\ham(\Vmat_{12}\zr_{12}))+\Vmat_{12}\tp\Pmat
		\begin{bmatrix}			0\\
			\Imat\end{bmatrix}
		(\Imat-\projmor_2^h)\patz2\ham(z_{12}).\end{equation*}
	Together with \eqref{eq:product_rule_for_squared_norm} and using the Cauchy--Schwarz and Young's inequalities, the Lipschitz continuity of $\nabla\ham$, and the fact that $\Vmat_{12}\tp\Vmat_{12}=\Imat$ implies $\norm{\theta}=\norm{\Vmat_{12}\thetar}=\norm{\thetar}$, this yields
	\begin{align*}		\frac{\mathrm{d}}{\dt}(\norm{\thetar}^2) &\le 2\norm{\thetar}\norm{\Pmat}(\norm{\nabla\ham(z_{12})-\nabla\ham(\Vmat_{12}\zr_{12})}+\norm{(\Imat-\projmor_2^h)\patz2\ham(z_{12})})\\
		&\le 2\norm{\thetar}\norm{\Pmat}(K\norm{\differror}+\norm{(\Imat-\projmor_2^h)\patz2\ham(z_{12})})\\
		&\le 2\norm{\thetar}\norm{\Pmat}(K(\norm{\rho}+\norm{\theta})+\norm{(\Imat-\projmor_2^h)\patz2\ham(z_{12})})\\
		&= 2K\norm{\Pmat}\norm{\thetar}^2+2\norm{\Pmat}(K\norm{\rho}+\norm{(\Imat-\projmor_2^h)\patz2\ham(z_{12})})\norm{\thetar}\\
		&\le 2K\norm{\Pmat}\norm{\thetar}^2+\norm{\Pmat}(K(\norm{\rho}^2+\norm{\thetar}^2)+\norm{(\Imat-\projmor_2^h)\patz2\ham(z_{12})}^2+\norm{\thetar}^2)\\
		&= C_1\norm{\thetar}^2+\norm{\Pmat}(K\norm{\rho}^2+\norm{(\Imat-\projmor_2^h)\patz2\ham(z_{12})}^2).\end{align*}
	Integrating over $[0,t]$ for some $t\in (0,T]$ and using a Gronwall-like inequality \cite[Thm.~1.3.4]{Pac98}, we obtain
	\begin{align*}		\norm{\thetar(t)}^2 &\le \int_0^t\norm{\Pmat}(K\norm{\rho(s)}^2+\norm{(\Imat-\projmor_2^h)\patz2\ham(z_{12}(s))}^2)e^{C_1(t-s)}\,\ds\\
		&\le e^{C_1T}\int_0^T\norm{\Pmat}(K\norm{\rho(s)}^2+\norm{(\Imat-\projmor_2^h)\patz2\ham(z_{12}(s))}^2)\,\ds.\end{align*}
	Using this bound, we get
	\begin{align*}		&\norm{\differror}_{L^2(0,T;\R^{n_1+n_2})}^2 = \norm{\rho}_{L^2(0,T;\R^{n_1+n_2})}^2+\norm{\theta}_{L^2(0,T;\R^{n_1+n_2})}^2\\
		&\le \norm{\rho}_{L^2(0,T;\R^{n_1+n_2})}^2+Te^{C_1T}\int_0^T\norm{\Pmat}(K\norm{\rho(s)}^2+\norm{(\Imat-\projmor_2^h)\patz2\ham(z_{12}(s))}^2)\,\ds\\
		&= (1+K\norm{\Pmat}Te^{C_1T})\norm{\rho}_{L^2(0,T;\R^{n_1+n_2})}^2+\norm{\Pmat}Te^{C_1T}\lVert (\Imat-\projmor_2^h)\patz 2 \ham(z_1, z_2)\rVert_{L^2(0,T;\R^{n_1+n_2})}^2,\end{align*}
	which shows the error bound for the differential variables.
	To obtain the bound for the algebraic variables, we use \eqref{eq:z3_in_terms_of_z1_and_z2} to obtain
	\begin{align*}		\norm{z_3(t)-\zr_3(t)} &\le C_2\norm{(\patz 2\ham(z_{12}(t))-\projmor_2^h\patz2\ham(\Vmat_{12}\zr_{12}(t)))}\\
		&\le C_2(\norm{(\Imat-\projmor_2^h)\patz2\ham(z_{12}(t))}+K\norm{\differror(t)})\end{align*}
	and, using the previously derived bound for $\norm{\differror}_{L^2(0,T;\R^{n_1+n_2})}^2$,
	\begin{align*}		\norm{z_3-\zr_3}_{L^2(0,T;\R^{n_3})}^2 \le &C_2K(1+K\norm{\Pmat}Te^{C_1T})\norm{\rho}_{L^2(0,T;\R^{n_1+n_2})}^2\\
		&+C_2(1+K\norm{\Pmat}Te^{C_1T})\lVert (\Imat-\projmor_2^h)\patz 2 \ham(z_1, z_2)\rVert_{L^2(0,T;\R^{n_1+n_2})}^2.\qedhere\end{align*}
\end{proof}
Similar to \cite[Thm.~2.2]{ChaBG16}, the error bounds in \Cref{thm:error_bound} consist of two terms: The first one scales with the error obtained by projecting the differential variables onto the column space of $\Vmat_{12}$, whereas the second one scales with the error corresponding to the projection of $\patz 2 \ham(z_1, z_2)$ onto the column space of $\Vmat_2$.

\subsection{Practical realization of time discretization}\label{subsec:time-discretization}

Although the test function $\phi^\tauh$ in~\eqref{eq:model-discrete} is such that the integral on the left-hand side can be computed exactly using suitable quadrature rules, this does not necessarily hold true for the integral on the right-hand side, since $\jbf$, $\rbf$, and $\bbf$ may be nonlinear.
Hence, to realize \Cref{scheme:full-discretization}, the right-hand side needs to be approximated appropriately.
To formalize this, we consider a general quadrature formula on $[0,1]$ with $\quadnodes \in \N$ quadrature nodes and quadrature weights $\omega_\ell$, $\ell = 1,\dots,\quadnodes$.
We assume that the weights $\omega_\ell$ are positive with $\sum_{\ell=1}^{\quadnodes} \omega_\ell = 1$.
Now, for $j=1,\dots,m$, consider quadrature formulas $Q_j \colon C([t_{j-1},t_{j}]) \to \R$ of the form
\begin{equation*}    Q_j(f) = (t_{j} - t_{j-1}) \sum_{\ell=1}^{\quadnodes} \omega_\ell f(\tau_{j,\ell})\end{equation*}
with the quadrature nodes $\tau_{j,\ell} \in [t_{j-1}, t_{j}]$, $\ell = 1,\dots,\quadnodes$.
Replacing the integral in the right-hand side of~\eqref{eq:model-discrete} by the sum of the quadrature formulas $Q_j$, we arrive at the model
\begin{equation}\label{eq:model-discrete-quadrature}
\begin{multlined}
    \int_{0}^{T}
    \Big\langle
        \Bigtextvec{\patz 1 \ham(z_{1}^\tauh,z_{2}^\tauh) \\ \pat z_{2}^\tauh \\ 0}
        ,
        \phi^\tauh
    \Big\rangle _{X^*, X}
    \dt
    \\
    \qquad=
    \sum_{j=1}^{m}
    Q_j
    \bigg(
    \Big\langle
        (\jbf - \rbf)
        \Bigtextvec{
            \pat z_{1}^\tauh
            \\
            \proj_2^\tau \proj_2^h \patz 2 \ham(z_{1}^\tauh, z_{2}^\tauh)
            \\
            z_{3}^\tauh
            }
        ,
        \phi^\tauh
    \Big\rangle _{X^*, X}
    +
    \big\langle
        \bbf(z^\tauh) u
        ,
        \phi^\tauh
    \big\rangle _{X^*,X}
    \bigg)
\end{multlined}\end{equation}
for all $\phi^\tauh \in \test^\tauh$.
As before, since the test spaces $\test_i^\tauh$ contain discontinuous functions that are nonzero only in one of the intervals $[t_{j-1}, t_{j}]$, system~\eqref{eq:model-discrete-quadrature} is equivalent to the time stepping scheme
\begin{equation}\label{eq:model-discrete-quadrature-localized}
\begin{multlined}
    \int_{t_{j-1}}^{t_{j}}
    \Big\langle
        \Bigtextvec{\patz 1 \ham(z_{1}^\tauh,z_{2}^\tauh) \\ \pat z_{2}^\tauh \\ 0}
        ,
        \phi^\tauh
    \Big\rangle _{X^*, X}
    \dt
    \\
    \qquad=
    Q_j
    \bigg(
    \Big\langle
        (\jbf - \rbf)
        \Bigtextvec{
            \pat z_{1}^\tauh
            \\
            \proj_2^\tau \proj_2^h \patz 2 \ham(z_{1}^\tauh, z_{2}^\tauh)
            \\
            z_{3}^\tauh
            }
        ,
        \phi^\tauh
    \Big\rangle _{X^*, X}
    +
    \big\langle
        \bbf(z^\tauh) u
        ,
        \phi^\tauh
    \big\rangle _{X^*,X}
    \bigg)
\end{multlined}\end{equation}
for all $\phi^\tauh \in \test^\tauh$ and $j=1,\dots,m$, where the $m$ nonlinear equations can be solved one after the other.

As with \Cref{prop:energy-balance-discrete}, we can show that solutions of~\eqref{eq:model-discrete-quadrature-localized} satisfy the discrete energy balance and dissipation inequality
\begin{align}	\ham(z_1^\tauh(t_{j}), &z_2^\tauh(t_{j})) - \ham(z_1^\tauh(t_{j-1}), z_2^\tauh(t_{j-1})) \nonumber \\
    &
    =
    Q_j
    \bigg(
    -
    \Big\langle
        \rbf
        \Bigtextvec{
            \pat z_{1}^\tauh
            \\
            \proj_2^\tau \proj_2^h \patz 2 \ham(z_{1}^\tauh, z_{2}^\tauh)
            \\
            z_{3}^\tauh
        }
        ,
        \Bigtextvec{
            \pat z_{1}^\tauh
            \\
            \proj_2^\tau \proj_2^h \patz 2 \ham(z_{1}^\tauh, z_{2}^\tauh)
            \\
            z_{3}^\tauh
        }
    \Big\rangle _{X^*, X}
    +
    \langle
        y^\tauh
        ,
        u
    \rangle _{U^*, U}
    \bigg)
    \label{eq:energy-balance-discrete-localized}
    \\
    &
    \leq
    Q_j
    \Big(
    \langle
        y^\tauh
        ,
        u
    \rangle _{U^*, U}
    \Big)
    \label{eq:dissipation-discrete-localized}\end{align}
for all $j=1,\dots,m$, where $y^\tauh$ is defined as in \Cref{prop:energy-balance-discrete}.

\begin{remark}\label{rem:no-z2-collocation}
    As mentioned before, if no $z_2$ component is present, then~\eqref{eq:model-discrete-quadrature-localized} can be interpreted as a collocation method, see also~\cite{EggHS21}.
    Essentially, this is due to the fact that both sides of~\eqref{eq:model-discrete-quadrature-localized} can be expressed as weighted averages of point evaluations of the integrands at points $t^{\ell}_{j} \in [t_{j-1}, t_{j}]$, $\ell \in \N$.
\end{remark}

Before we turn our attention to the numerical experiments, let us remark on the computational aspects of the projective terms.
\begin{remark}\label{rem:computational-aspects-projections}
    As seen in \Cref{rem:removing-z2}, the~$z_2$ variable can be eliminated from the model~\eqref{eq:model-weak} by doubling the corresponding dimension.
    Since no projections are necessary if~$z_2 = \bullet$, this alternative formulation enables us to fairly compare the additional computational effort associated with the projective terms.

    We have demonstrated in \Cref{subsec:space} that the spatial projection is implicitly encoded in the space discrete energy functional.
    This, however, also incorporates mass matrices in the second equation of~\eqref{eq:time-continuous-spatially-discrete-coefficients-with-mass-matrices}.
    Hence, to evaluate the operators~$\jbf^h$, $\rbf^h$ and $\bbf^h$, a typically sparse linear system needs to be solved.
    For the time discretization, the use of $L^2$-orthonormal basis functions eliminates additional linear systems and, hence, computing the temporal projection amounts to a series of tensor multiplications.

    If, in contrast, the~$z_2$ variable is eliminated from the model, then the nonlinear system~\eqref{eq:model-discrete-quadrature-localized} is more expensive to solve due to the increased system size.

    Which of these situations is favorable likely depends on the application.
    We expect the former approach to be beneficial if the dimension of~$Z^h_2$ is large compared to the dimensions of~$Z^h_1$ and~$Z^h_3$, and the latter approach to be beneficial if the linear systems necessary for the spatial discretization are a main driver of the computation cost.
\end{remark}

\section{Numerical Examples}\label{sec:numerics}
We present three examples where \Cref{scheme:full-discretization} is used in different ways.
For a nonlinear circuit example, we analyze the time discretization.
For the Cahn--Hilliard equation, on the other hand, we use the scheme for the spatial discretization and model reduction.
Finally, we consider a doubly nonlinear parabolic equation, where the introduced projections 
are crucial.
We start, however, with a short discussion of the implementation.

\subsection{Implementation details}\label{subsec:implementation-details}

\subsubsection*{Polynomials in time}
As described in \Cref{sec:discretization}, we use piecewise polynomials for the time-discrete spaces~$\ansatz_i^\tauh$.
In particular, we use scaled Legendre polynomials defined on the intervals $[t_{j-1}, t_{j}]$, $j=1,\dots,m$ as the bases of these spaces.
The scaling factors are chosen such that the polynomials are $L^2$-orthonormal.

\subsubsection*{Projections}
As seen in \Cref{subsec:space,subsec:model-reduction}, the spatial projection $\proj^h$ is implicitly included by considering the appropriate Hamiltonian.
To compute the temporal projection $\proj^\tau$ of a function $f \colon [t_{j-1},t_{j}] \to \R$, we use Gauß quadrature with $\projnodes \in \N$ nodes to approximate the integrals $\int_{t_{j-1}}^{t_{j}} f \phi_{p}^\tau \dt$ for the scaled Legendre polynomials $\phi_{p}^\tau$, $p=1,\dots,k+1$.
As before, $k$ denotes the polynomial degree and $\phi_{p}^\tau$ are assumed to form an $L^2$-orthonormal system such that the integrals above correspond to the coefficients w.r.t.~the basis $\{ \phi_p^\tau ~|~ p = 1,\dots,k+1\}$.
Since the first component of the left-hand sides of~\eqref{eq:model-discrete} and~\eqref{eq:model-discrete-quadrature-localized} also corresponds to a temporal projection, we proceed similarly and use $\projnodes$ nodes for the Gauß quadrature there.

\subsubsection*{Nonlinear systems}
To solve the nonlinear system~\eqref{eq:model-discrete-quadrature-localized}, we use Newton's method.
The derivatives required are computed automatically using \textsf{JAX}~\cite{bradbury18-jax} and the linear systems are solved using \textsf{Lineax}~\cite{rader23-lineax}.
In the first time step, the constant one vector is used as a starting value.
In subsequent steps, the numerical solution of the previous time step is used.
We perform ten Newton steps, which leads to convergence of the iteration for all our examples.

\subsubsection*{Temporal convergence}
We use the method of manufactured solutions to analyze the temporal convergence of our scheme.
Given a target solution $\zm = (\zm_1, \zm_2, \zm_3)$ and a fixed control signal $u$, we find $\manuin$ such that $\zm$ is the solution of
\begin{equation*}	\Big\langle
	\bigtextvec{
		\patz 1 \ham(z_1, z_2)
		\\
		\pat z_2
		\\
		0
	}
	,
	\phi
	\Big\rangle _{X^*,X}
	=
	\Big\langle
	(\jbf - \rbf)
	\bigtextvec{
		\pat z_1
		\\
		\patz 1 \ham(z_1, z_2)
		\\
		0
	}
	+
	\bbf(z) u
	+
	\manuin
	,
	\phi
	\Big\rangle _{X^*,X}\end{equation*}
and then use $\manuin$ as an additional control input to~\eqref{eq:model-discrete-quadrature-localized}.
We then visualize the relative error in the non-algebraic variables given by
\begin{equation}\label{eq:non-algebraic-state-error}
	\errorstate^{\text{non-alg}}(\tau)
	\coloneqq
	\frac{
		\max_{t \in [0,T]}\norm{(\zm_1^h(t), \zm_2^h(t)) - (z_1^\tauh(t), z_2^\tauh(t))}
	}{
		\max_{s \in [0,T]} \norm{(\zm_1^h(s), \zm_2^h(s))}
	},\end{equation}
the relative error in all variables given by
\begin{equation}\label{eq:full-state-error}
	\errorstate(\tau)
	\coloneqq
	\frac{
		\max_{t \in [0,T]}\norm{(\zm_1^h(t), \zm_2^h(t), \zm_3^\tauh(t)) - (z_1^\tauh(t), z_2^\tauh(t), z_3^\tauh(t))}
	}{
		\max_{s \in [0,T]} \norm{(\zm_1^h(s), \zm_2^h(s), \zm_3^h(s))}
	},\end{equation}
where in both cases the maximization is performed on a reference grid with mesh size $\tau_{\text{ref}} = 2^{-3} \tau_{\text{min}}$ with $\tau_{\text{min}}$ being the smallest tested mesh size,
and the nodal relative error in all variables given by
\begin{equation}\label{eq:nodal-state-error}
    \errorstatenodal(\tau)
    \coloneqq
    \frac{
        \max_{j=0,\dots,m}
        \norm{
            \zm^h(t_j) - z^\tauh(t_j)
        }_{X}
    }{
        \max_{s=0,\dots,m}
        \norm{
            \zm^h(t_s)
        }_{X}
    }.\end{equation}

\subsubsection*{Energy balance verification}
To verify the energy balance~\eqref{eq:energy-balance-discrete-localized}, we visualize the relative error
\begin{equation}\label{eq:relative-energy-balance-error}
	\errorenergy(t_j)
	\coloneqq
	\frac{
		\left|
		\ham((z_1^\tauh,z_2^\tauh)(t_{j})) - \ham((z_1^\tauh,z_2^\tauh)(t_{j-1}))
		-
		Q_j(
        -
		\langle \rbf (\zeta), \zeta \rangle
		+
		\langle \bbf(z^\tauh) u, \zeta \rangle
		)
		\right|
	}{
		\max_{s=1,\dots,m}
		|
		\ham((z_1^\tauh,z_2^\tauh)(t_{s})) - \ham((z_1^\tauh,z_2^\tauh)(t_{s-1}))
		|
	}\end{equation}

with $\zeta = (\pat z_{1}^\tauh, \proj_2^\tau \proj_2^h \patz 2 \ham(z_{1}^\tauh, z_{2}^\tauh), z_{3}^\tauh)$.

\subsection{Nonlinear circuit model}\label{sec:numerics:circuit}

We consider a ladder network as depicted in \Cref{fig:circuit}.

\begin{figure}
	\begin{circuitikz}[european]
		\draw[thick] (0,0) to[sinusoidal voltage source] (0,3) -- (3,3) to[C] (3,0) -- (0,0);
		\draw[thick] (3,3) to[R] (5,3) to[american inductor] (7,3) to[C] (7,0) -- (3,0);
		\draw[thick] (7,3) to[R] (9,3) to[american inductor] (11,3) to[R] (11,0) -- (7,0);
		\draw[thick] (3,3) -- (2,2.5) node[ground]{};
	\end{circuitikz}
	\caption{Ladder network for the example considered in \Cref{sec:numerics:circuit}.}
	\label{fig:circuit}
\end{figure}

Following the modified nodal analysis as in~\cite{GerHRS21}, the governing equations may be written as

\begin{equation}\label{eq:circuit}
    \begin{aligned}
    \Ac \dot{q}_\mathrm{C} + \AS \iS + \AR \Gmat \AR\tp \phi+\AL\nabla \HL(\psiL)  & = 0,\\
    -\AL\tp\phi+\dot{\psi}_{\mathrm{L}} &= 0,\\
    \Ac\tp \phi - \nabla \HC(\qC) & = 0, \\
    \AS\tp \phi - \uS & = 0,
    \end{aligned}\end{equation}

where $\qC\colon [0,T]\to\R^2$ contains the charges at the capacitors, $\iS,\uS\colon [0,T]\to\R$ the current and voltage at the voltage source, respectively, $\phi\colon [0,T]\to\R^5$ the vertex potentials, and $\psiL\colon [0,T]\to\R^2$ the magnetic fluxes at the inductors. $\HL,\HC\colon \R^2\to\R$ denote the total energies of the inductors and capacitors, respectively, and $\Gmat\in\R^{3,3}$ is a diagonal matrix containing the conductances, i.e., the reciprocals of the resistances. Moreover, $\Ac\in\R^{5,2}$, $\AS\in\R^{5,1}$, $\AR\in\R^{5,3}$, $\AL\in\R^{5,2}$ are the parts of the incidence matrix corresponding to the edges of the capacitors, voltage source, resistors, and inductors, respectively.
For the given circuit, these matrices are given by
\begin{equation*}	\Ac =
	\begin{bmatrix}		0 & 0\\
		0 & -1\\
		0 & 0\\
		0 & 0\\
		1 & 1\end{bmatrix}
	,\quad \AS =
	\begin{bmatrix}		0\\
		0\\
		0\\
		0\\
		-1\end{bmatrix}
	,\quad \AR =
	\begin{bmatrix}		1 & 0 & 0\\
		0 & -1 & 0\\
		0 & 1 & 0\\
		0 & 0 & -1\\
		0 & 0 & 1\end{bmatrix}
	,\quad \AL =
	\begin{bmatrix}		-1 & 0\\
		1 & 0\\
		0 & -1\\
		0 & 1\\
		0 & 0\end{bmatrix}
	.\end{equation*}
In the following, we set
\begin{equation*}	G = \eye_3,\quad
	\HL(\psiL) = \tfrac12\, (\norm{\psiL}^2 + \norm{\psiL}^4),\quad
	\HC(\qC) = \tfrac12 \norm{\qC}^2,\quad \uS(t)=\sin(t).\end{equation*}
The index of the system is two, cf.~\cite[Sec.~4]{SchMMV23}.
It can be shown that the initial values have to satisfy the consistency condition
\begin{equation*}	\begin{bmatrix}		0 & -1 & 0 & 0 & 0 & 0 & 0\\
		0 & 0 & 1 & -1 & 0 & 0 & 0\\
		0 & 0 & 0 & 0 & -1 & 1 & 0\\
		0 & 0 & -1 & 1 & 1 & -1 & 1\\
		0 & 0 & 0 & 0 & 0 & -1 & 0\\
		1 & 0 & 0 & 0 & 0 & -1 & 0\\
		0 & 0 & 1 & 0 & 0 & -1 & 0\\\end{bmatrix}
	\begin{bmatrix}		[\qC(0)]_1\\
		\phi(0)\\
		\iS(0)\end{bmatrix}
	=
	\begin{bmatrix}		-1 & 0 & 0\\
		0 & -1 & 0\\
		0 & 1 & 0\\
		1 & 0 & 0\\
		0 & 0 & 0\\
		0 & 0 & 0\\
		0 & 0 & -1\end{bmatrix}
	\begin{bmatrix}		\nabla \HL(\psiL)\\
		[\qC(0)]_2\end{bmatrix}
	-\begin{bmatrix}		0\\
		0\\
		0\\
		\dot{u}_{\mathrm{S}}(0)\\
		\uS(0)\\
		0\\
		0\end{bmatrix}
	.\end{equation*}
We set all initial values to $0$ except for $\iS(0)=-\dot{u}_{\mathrm{S}}(0)=-\cos(0)=-1$.
As shown in~\cite{AltS25}, system~\eqref{eq:circuit} can be written as
\begin{equation*}    \vec{
        \nabla \HC(\qC)
        \\
        \dot{\psi}_{\mathrm{L}}\\
        0
        \\
        0
    }
    = \underbrace{
    \vec{
        \phantom{-}0 & \phantom{-}0 & \phantom{-}0 & \Ac\tp \\
        \phantom{-}0 & \phantom{-}0 & \phantom{-}0 & \AL\tp \\
        \phantom{-}0 & \phantom{-}0 & \phantom{-}0 & \AS\tp \\
        -\Ac & -\AL & -\AS & -\AR \Gmat \AR\tp
    }}_{=\vcentcolon\Amat}
    \vec{
        \dot{q}_{\mathrm{C}}
        \\
        \nabla \HL(\psiL)\\
        \iS
        \\
        \phi
    }
    +\underbrace{
    \vec{0 \\ 0\\ -1 \\ 0}}_{=\vcentcolon\Bmat}\uS\end{equation*}
or, equivalently, in the form~\eqref{eq:model-strong} with $z_1 = \qC$, $z_2 = \psiL$, $z_3 = (\iS, \phi)$, $\jbf = \tfrac12 (\Amat - \Amat\tp)$, $\rbf = -\tfrac12 (\Amat + \Amat\tp)$, and $\ham(z_1,z_2) = \HC(\qC)+\HL(\psiL)$.

The number of quadrature nodes $\projnodes$ used in the temporal projection was chosen as $\projnodes = 2k$, where $k$ is the polynomial degree used in the temporal ansatz space.
This choice is motivated by $\nabla \ham_C$ being a third order polynomial so that the maximal polynomial degree of the projection terms is $3k + k-1 = 4k-1$, i.e.,~$3$ times degree $k$ plus the degree of the test function~$k-1$.
Furthermore, we used $\quadnodes = k+1$ for all of our experiments.

To analyze the temporal convergence of the scheme, we consider the manufactured solution 
and control signal 
\begin{gather*}    q_C(t) = \vec{\cos(t) \\ \cos(t)} \in \RR^2,
    \qquad
    \psi_L(t) = \vec{\cos(t) \\ \cos(t)} \in \RR^2,
    \qquad
    \imath_S(t) = \sin(t),
    \\
    \phi(t) = \vec{\sin(t) \\ \vdots \\ \sin(t)} \in \RR^5,
    \qquad
    u_S(t) = \sin(t) - 1.\end{gather*}
For the computation of the relative error in the energy balance and the visualization of the Hamiltonian, we chose the initial condition $z_0 = (q_{C,0}, \psi_{L,0}, \imath_{S,0}, \phi_0)$ and again the control signal $u_S(t) = \sin(t)$ with $q_{C,0} = \psi_{L,0} = 0$, $\imath_{S,0} = -1$ and $\phi_0 = 0$.

In \Cref{fig:nonlinear_circuit_convergence_all,fig:nonlinear_circuit_convergence_nonalg}, the convergence history in time is shown, once for all and once for the non-algebraic variables only.
In both experiments, several polynomial degrees $k$ with $\projnodes = 2k$ and $\quadnodes = k+1$ are used.
For odd polynomial degree~$k$, we observe the order of convergence $k+1$ for the non-algebraic variables $q_C$ and $\psi_L$ and order~$k$ for the algebraic variables $i_S$ and $\phi$.
For even polynomial degrees, the observed orders match that of the preceding odd degree.
\Cref{fig:nonlinear_circuit_energybalance} visualizes the relative error in the energy balance for several choices of $k$ for~$\tau = 10^{-2}$.
As expected, we observe that our scheme satisfies the energy balance up to machine precision in each time step.

\begin{figure}
    \centering
    \begin{subfigure}[t]{.32\textwidth}
        \centering
        \scalebox{0.39}{\input{figures/nonlinear_circuit_varying_degree.pgf}}
        \vspace{-6pt}
        \caption{all variables}
        \label{fig:nonlinear_circuit_convergence_all}
    \end{subfigure}
    \hfill
    \begin{subfigure}[t]{.32\textwidth}
        \centering
        \scalebox{0.39}{\input{figures/nonlinear_circuit_varying_degree_nonalg.pgf}}
        \vspace{-6pt}
        \caption{non-algebraic variables}
        \label{fig:nonlinear_circuit_convergence_nonalg}
    \end{subfigure}
    \hfill
    \begin{subfigure}[t]{.32\textwidth}
        \centering
        \scalebox{0.39}{\input{figures/nonlinear_circuit_energybalance.pgf}}
        \vspace{-6pt}
        \caption{relative error in energy balance}
        \label{fig:nonlinear_circuit_energybalance}
    \end{subfigure}
    \caption{Temporal convergence of \Cref{scheme:full-discretization} applied to the nonlinear circuit~\eqref{eq:circuit} (left, middle) and relative error in the energy balance with $\tau = 10^{-2}$ (right) for several polynomial degrees $k$ with $\projnodes=2k$ and $\quadnodes=k+1$.}
    \label{fig:nonlinear_circuit}
\end{figure}

\subsection{Cahn--Hilliard equation}\label{sec:numerics:CH}

The Cahn--Hilliard equations~\cite{CahH58,EllS86} model phase separation in binary fluids.
To state the model, let $\Omega \subseteq \R^2$ be a bounded Lipschitz domain with boundary $\partial \Omega$.
Then, the Cahn--Hilliard equation on $[0,T] \times \Omega$ reads as

\begin{subequations}\label{eq:cahn-hilliard-strong}
\begin{align}    \pat v - \sigma \Delta w
    &= 0,
    \\
    - \eps \Delta v + \eps^{-1} W^\prime(v)
    &= w.\end{align}\end{subequations}

Here, $v\colon [0,T] \times \Omega \to \R$ is the difference of the phase fractions of the two components, $w\colon [0,T] \times \Omega \to \R$ represents the chemical potential, $\eps$ denotes the so-called interaction length, $\sigma$ acts as a mobility parameter, and $W$ is an energy potential with typically two minima.
We augment~\eqref{eq:cahn-hilliard-strong} with the boundary conditions

\begin{equation}\label{eq:cahn-hilliard-boundary-conditions}
    \nabla w\tp \normal = u_1,
    \qquad
    \nabla v\tp \normal = u_2\end{equation}

for given functions $u_1, u_2 \in C([0,T], L^2(\partial \Omega))$ and outer unit normal $\normal$ along~$\partial \Omega$.
In contrast to~\cite{EggHS21}, we do not eliminate the chemical potential $w$ from~\eqref{eq:cahn-hilliard-strong} and leave the algebraic equation in our model. Note that the system corresponds to an index-1 system in the sense that a standard spatial discretization would yield a differential--algebraic equation of index~1.
Toward~\eqref{eq:model-weak}, let us define $X_1 \coloneqq X_3 \coloneqq H^1(\Omega) \embeds \lebesgue^2(\Omega) \eqcolon Z_1 \eqcolon Z_3$.
The system is associated with the energy
\begin{equation*}    \ham\colon X_1 \to \R,\quad
    v \mapsto \int_\Omega \frac\eps2\, \nabla v\tp \nabla v + \frac1\eps\, W(v) \dx,\end{equation*}
which is well defined on $H^1(\Omega) \embeds \lebesgue^4(\Omega)$ and satisfies
\begin{equation*}    \langle
        \nabla \ham(v)
        ,
        \phi
    \rangle _{X_1^*, X_1}
    =
    \int_\Omega
        \eps\, \nabla v\tp \nabla \phi
        +
        \eps^{-1} \phi W'(v)
    \dx\end{equation*}
for all $v,\phi \in X_1$.
After testing~\eqref{eq:cahn-hilliard-strong} with functions $\psi, \phi \in H^1(\Omega)$ and applying Green's first identity, we arrive at

\begin{align*}    \int_\Omega
        \psi \pat v
        +
        \sigma \nabla w\tp \nabla \psi
    \dx
    &
    =
    \int_{\partial \Omega}
        \sigma \psi u_1
    \ds,
    \\
    \int_\Omega
        \eps \nabla v\tp \nabla \phi
        +
        \eps^{-1} \phi W'(v)
    \dx
    &
    =
    \int_\Omega
        \phi w
    \dx
    +
    \int_{\partial \Omega}
        \eps \phi u_2
    \ds.\end{align*}

This formulation is to be understood pointwise in time.
Now, define $\jbf,\rbf \colon X \to X^*$ via

\begin{equation*}    \Big\langle
        \jbf \bigtextvec{\pat v \\ w}
        ,
        \bigtextvec{\phi \\ \psi}
    \Big\rangle _{X^*,X}
    =
    \int_\Omega
        \phi w - \psi \pat v
    \dx
    ,\qquad
    \Big\langle
        \rbf \bigtextvec{\pat v \\ w}
        ,
        \bigtextvec{\phi \\ \psi}
    \Big\rangle _{X^*,X}
    =
    \int_\Omega
        \sigma \nabla w\tp \nabla \psi
    \dx.\end{equation*}

With $u = (u_1, u_2)$ we further define $\bbf \in \boundedlinear(L^2(\partial \Omega) \times L^2(\partial \Omega), X^*)$ by

\begin{equation*}    \Big\langle
        \bbf u
        ,
        \bigtextvec{\phi \\ \psi}
    \Big\rangle _{X^*,X}
    =
    \int_{\partial \Omega}
        \sigma \psi u_1 + \eps \phi u_2
    \ds.\end{equation*}

Since $\jbf$ and $\rbf$ satisfy~\eqref{eq:J-R-properties}, we arrive with $z = (v, \bullet, w)$ at a system of the form~\eqref{eq:model-weak}.

For the space discretization of the Cahn--Hilliard model~\eqref{eq:cahn-hilliard-strong}, we focus on the unit square $\Omega = (0,1)^2$.
We use first-order Lagrange finite elements defined on a uniform triangulation of the domain to define the approximation spaces $X_i^h$.
To define the elements, we utilize the Python package \textsf{fenics-basix}~\cite{scroggs22-basix,scroggs22-construction}.
We use the constant mesh width $h=1/10$ for all of our experiments.

We consider the double well potential $W(v) = \frac14 (v^2 - 1)^2$ and homogeneous Neumann boundary conditions $u_1 = u_2 = 0$.
The number of quadrature nodes $\projnodes$ used in the temporal projection was chosen as $\projnodes = 2k$ per default, where $k$ is the polynomial degree used in the temporal ansatz space.
This choice is motivated by $W'$ being a third order polynomial so that the maximal polynomial degree of the projection terms is $3k + k-1 = 4k-1$.
Furthermore, we observed that the choice $\quadnodes = 2k$ leads to better convergence of Newton's method in the time stepping scheme, which is why we used this setting for all of our experiments.
For the initial condition of~\eqref{eq:cahn-hilliard-strong}, we initialize $v$ as fractal noise projected onto the spatial ansatz space and then compute $w$ satisfying the second equation in~\eqref{eq:cahn-hilliard-strong}.
To obtain the matrices $\Vmat_1\in \R^{n_1, \nr_1}$ and $\Vmat_3 \in \R^{n_3, \nr_3}$ for the ROM as in \Cref{subsec:model-reduction}, we use proper orthogonal decomposition~\cite{sirovich87-turbulence,holmes12-turbulence}.
For this, we simulate the FOM with step size $\tau = 10^{-2}$ and choose the reduced order dimensions as $\nr_1 = \nr_3 = 5$.

In \Cref{fig:cahn_hilliard_state}, the difference of the phase fractions $v$ is shown at various time points for the full and reduced order models.
For the temporal discretization, we used piecewise polynomials with degree $k=3$ for $z_1$ and $z_3$ and $\quadnodes = \projnodes = 6$ and $\tau = 10^{-2}$.
We observe that our scheme successfully reproduces the key characteristics of the model.
The corresponding space discrete and reduced order Hamiltonians are shown in \Cref{fig:cahn_hilliard_hamiltonian}.
We observe that the behavior of the reduced order Hamiltonian $\hamr$ is similar to the one of the full order Hamiltonian $\ham^h$.
\Cref{fig:cahn_hilliard_energybalance} depicts $\errorenergy$ 
for the spatially discrete FOM and ROM.
We observe that the energy balance is satisfied up to machine precision in each time step if the appropriate number of quadrature nodes $\projnodes$ is used for computing the projection $\proj^\tau_2$.
As outlined before, for our choice of the potential $W$ the appropriate choice is $\projnodes = 2k$.
Reducing $\projnodes$ drastically increases the error $\errorenergy$, which stems from the first component of the left-hand side of~\eqref{eq:model-discrete-quadrature-localized}.
\Cref{fig:cahn_hilliard_convergence} shows the temporal convergence of the time stepping scheme~\eqref{eq:model-discrete-quadrature-localized}.
For the manufactured solution, we choose $\zm_1 = v$ as
\begin{equation*}    \zm_1(t,(x_1,x_2))
    \coloneqq \big(\tfrac{1}{10} \sin(t) + 1\big) \Big(2 \exp\big(-25\,(x_1-\tfrac12)^2 - 25\,(x_2-\tfrac12)^2\big) - 1\Big)\end{equation*}
and determine $\zm_3=w$ according to the algebraic equation.
For the computation of the error~\eqref{eq:non-algebraic-state-error}, we then project both $\zm_1$ and $\zm_3$ to the ansatz space, so that the space discretization error does not influence the convergence.

We observe that our scheme has order of convergence~$k$ in all variables and~$k+1$ in the non-algebraic variables.
When the error in the algebraic variable~$w$ is calculated by comparing to the temporal projection of the manufactured solution onto the piecewise polynomials of degree at most~$k-1$, the convergence order~$k+1$ can be observed also for this error.
This is in accordance to the results of~\cite{brunk23-stability}.

\begin{figure}
    \centering
    \begin{subfigure}[t]{.32\textwidth}
        \centering
        \scalebox{0.38}{\input{figures/cahn_hilliard_0.pgf}}
        \caption{$t=0$, full order model}
        \label{fig:cahn_hilliard_state_0}
    \end{subfigure}
    \hfill
    \begin{subfigure}[t]{.32\textwidth}
        \centering
        \scalebox{0.38}{\input{figures/cahn_hilliard_2.pgf}}
        \caption{$t=0.75$, full order model}
        \label{fig:cahn_hilliard_state_2}
    \end{subfigure}
    \hfill
    \begin{subfigure}[t]{.32\textwidth}
        \centering
        \scalebox{0.38}{\input{figures/cahn_hilliard_4.pgf}}
        \caption{$t=1.5$, full order model}
        \label{fig:cahn_hilliard_state_4}
    \end{subfigure}
    \\
    \begin{subfigure}[t]{.32\textwidth}
        \centering
        \scalebox{0.38}{\input{figures/cahn_hilliard_rom_0.pgf}}
        \caption{$t=0$, reduced order model}
        \label{fig:cahn_hilliard_rom_state_0}
    \end{subfigure}
    \hfill
    \begin{subfigure}[t]{.32\textwidth}
        \centering
        \scalebox{0.38}{\input{figures/cahn_hilliard_rom_2.pgf}}
        \caption{$t=0.75$, reduced order model}
        \label{fig:cahn_hilliard_rom_state_2}
    \end{subfigure}
    \hfill
    \begin{subfigure}[t]{.32\textwidth}
        \centering
        \scalebox{0.38}{\input{figures/cahn_hilliard_rom_4.pgf}}
        \caption{$t=1.5$, reduced order model}
        \label{fig:cahn_hilliard_rom_state_4}
    \end{subfigure}
    \caption{Difference of the phase fractions $v$ for the full and reduced order models of the Cahn--Hilliard equation~\eqref{eq:cahn-hilliard-strong} for several time points (top: full order model, bottom: reduced order model).
        In both cases, we use $k=3$, $\quadnodes = \projnodes = 6$, and $\tau=10^{-2}$ for the time discretization.
        The full order model was obtained with space mesh width $h=1/10$.}
    \label{fig:cahn_hilliard_state}
\end{figure}

\begin{figure}
    \centering

    \scalebox{0.55}{\input{figures/cahn_hilliard_energybalance.pgf}}

    \caption{Relative error in the energy balance for the full (left) and reduced order models (right) of the Cahn--Hilliard equation~\eqref{eq:cahn-hilliard-strong} for several $k$, $\projnodes$ and fixed $\quadnodes=2k$.
        In all cases, we use $\tau=10^{-2}$ for the time discretization.
        The full order model was obtained with mesh width $h=1/10$.}
    \label{fig:cahn_hilliard_energybalance}
\end{figure}

\begin{figure}
    \centering
    \begin{subfigure}[t]{.32\textwidth}
        \centering
        \scalebox{0.39}{\input{figures/cahn_hilliard_hamiltonian.pgf}}
        \vspace{-6pt}
        \caption{Space-discrete and reduced order energies $\ham^h$ and $\hamr$ for $k=3$ and $\quadnodes = \projnodes = 6$.}
        \label{fig:cahn_hilliard_hamiltonian}
    \end{subfigure}
    \hfill
    \begin{subfigure}[t]{.32\textwidth}
        \centering
        \scalebox{0.39}{\input{figures/cahn_hilliard_varying_degree.pgf}}
        \vspace{-6pt}
        \caption{Temporal convergence in all state variables for several polynomial degrees $k$ with $\quadnodes=\projnodes=2k$.}
        \label{fig:cahn_hilliard_convergence_all}
    \end{subfigure}
    \hfill
    \begin{subfigure}[t]{.32\textwidth}
        \centering
        \scalebox{0.39}{\input{figures/cahn_hilliard_varying_degree_nonalg.pgf}}
        \vspace{-6pt}
        \caption{Temporal convergence in the non-algebraic states for several polynomial degrees $k$ with $\quadnodes=\projnodes=2k$.}
        \label{fig:cahn_hilliard_convergence}
    \end{subfigure}
    \caption{Energy (left) and convergence (middle, right) for the Cahn--Hilliard equation~\eqref{eq:cahn-hilliard-strong} for spatial mesh width $h=1/10$.}
    \label{fig:cahn_hilliard_hamiltonian_and_convergence}
\end{figure}

\subsection{Doubly nonlinear parabolic equation}
As in~\cite{GieKT25}, let us consider a doubly nonlinear parabolic equation, which leads to a model of the form~\eqref{eq:model-weak} with $z_1 = z_3 = \bullet$.
In contrast to~\cite{GieKT25}, where numerically only a special case with vanishing projection was investigated, here we also consider cases in which the projections in \Cref{scheme:full-discretization} do not vanish.
To simplify the presentation, we focus on a one-dimensional spatial domain and neglect boundary controls.
To see how to formulate the more general case in our framework, the interested reader is referred to~\cite{GieKT25}.
Let $\Omega = (0, 1) \subseteq \RR$ and set~$Q \coloneqq (0,T) \times \Omega$ for $T > 0$.
We assume~$p > 1$, $q \geq 1$ and that~$\dnpa, \dnpb \colon \RR \to \RR$ are given by
\begin{equation*}    \dnpa(z) = \abs{z}^{q-1} z, \qquad
    \dnpb(w) = \abs{w}^{p-2} w.\end{equation*}
On $Q$ we consider the nonlinear scalar evolution equation
\begin{align}\label{eq:doubly-parabolic}
    \pat z = \pax \dnpb \big( \pax \dnpa(z) \big)\end{align}
augmented with the initial data~$z(0,x) = z_0(x)$ and the boundary conditions
\begin{equation*}    \dnpb\big(\pax \dnpa(z(t,0))\big) = 0, \qquad
    \dnpb\big(\pax \dnpa(z(t,1))\big) = 0.\end{equation*}
For~$q = 1$, our problem reduces to the \emph{$p$-Laplace equation}~\cite{heinonen06-nonlinear} and for~$p = 2$ it reduces to the \emph{porous medium equation}~\cite{vazquez07-porous}.
To write~\eqref{eq:doubly-parabolic} in the form~\eqref{eq:model-weak}, we consider~$\phi \in C^{\infty}(\overline Q)$, integrate by parts, and employ the boundary conditions to arrive at
\begin{align*}    \int_{\Omega}
        \pax \dnpb \big( \pax \dnpa(z) \big)
        \,
        \phi
    \dx
    & =
    -
    \int_{\Omega}
        \abs{\pax \dnpa(z)}^{p-2} \pax \dnpa(z)
        \,
        \pax \phi
    \dx.\end{align*}
Since~$p\geq \frac{2d}{d+1} > \frac{2d}{d+2}$, there is a continuous and dense embedding~\cite[Thm.~4.12]{adams75-sobolev} $X \coloneqq W^{1,p}(\Omega) \hookrightarrow \lebesgue^2(\Omega) \eqqcolon Z$.
System~\eqref{eq:doubly-parabolic} is associated with the energy functional
\begin{align*}    \mathcal{H}(z) \coloneqq  \tfrac{1}{q+1} \int_{\Omega} \abs{z}^{q+1} \dx,\end{align*}
which is well-defined for any~$z\in D\coloneqq \tilde{X}\coloneqq \lebesgue^{q+1}(\Omega) \cap X$ and has the continuous Fréchet derivative
\begin{align*}    \mathcal{H'}(z) = \dnpa(z) = \abs{z}^{q-1}z\end{align*}
arising from the porous medium part of the equation.

We consider $z_1 = z_3 = \bullet$ and $\bbf \coloneqq 0$ (there are no inputs in our model). Further we define~$\jbf, \rbf \colon X \to X^*$ as $\jbf \coloneqq 0$ and
\begin{align*}    \langle
        \rbf(v)
        ,
        \phi
    \rangle_{X^*,X}
    & \coloneqq
    \int_{\Omega}
        \abs{\pax v}^{p-2} \pax v
        \,
        \pax \phi
    \dx\end{align*}
for all~$v,\phi \in X$.

Note that the nonlinearity of~$\rbf$ stems from the~$p$-Laplace structure and that~$\rbf$ satisfies~\eqref{eq:J-R-properties}.
To see how inputs can be incorporated, see~\cite{GieKT25}.

For our numerical experiments, we consider $T = 0.1$, since the system quickly reaches an equilibrium.
Moreover, we consider the parameters~$p = 1.5$ and~$q \in \{1.5, 3\}$.

Note that neither of the projections in \Cref{scheme:full-discretization} vanish if~$p\neq 2$ and~$q \neq 1$.
We refer to~\cite{Kar26} for other choices of parameters and a more detailed discussion of the influence of the spatial discretization parameter.

For the space discretization, we use globally continuous, piecewise linear functions, i.e., hat functions~\cite[Ch.~3]{quarteroni94-numerical}.
Since the temporal projection is exact for~$p=2$ and~$q=3$ if~$\projnodes = 2k$ is used, we default to this choice for all of our experiments.
Furthermore, we use~$\quadnodes = 2k$ to improve the convergence of Newton's method in the time stepping scheme.

For the convergence experiments, we use the manufactured solution
\begin{equation*}    \zm(t,x) \coloneqq e^{-\dnptimescale t} \cos \left( \tfrac{4 \pi x}{L} \right),\end{equation*}
where~$\dnptimescale \coloneqq 50$ is a time scaling parameter.
In all other experiments, we use~$z_0(x) \coloneqq \zm(0, x) = \cos \left( \tfrac{4 \pi x}{L} \right)$ as initial condition.

The structure-preserving properties of our scheme are illustrated in \Cref{fig:doubly_nonlinear_parabolic_energybalance_fom} where the results of the energy balance experiment are shown for~$\tau = 2 \cdot 10^{-4}$ and various choices of~$k$.
It can be observed that the energy balance~\eqref{eq:model-discrete-quadrature-localized} is satisfied up to machine precision for the majority of the time horizon.
The relative error $\errorenergy$ is largest at the start of the time horizon for some parameter settings.
This is due to the Newton solver used in the time stepping scheme, see~\cite{slodicka02-robust} for notes on more specialized schemes for doubly nonlinear parabolic equations.

In \Cref{fig:doubly_nonlinear_parabolic_convergence}, the results of our convergence experiments are shown for the spatial discretization parameter~$h=\tfrac{1}{25}$.

For~$k=2$ and~$k=3$, we observe order of convergence $k+1$ until the relative error stagnates at around $10^{-8}$, while for~$k=4$ no convergence can be observed.
Similar observations can be made for the nodal superconvergence{, for which we refer to~\cite{Kar26}.}
In \Cref{fig:doubly_nonlinear_parabolic_convergence_no_projection}, the convergence behavior of \Cref{scheme:full-discretization} is shown when $\proj_2^\tau$ is omitted.
While the general convergence behavior seems unaltered, we observe that the error stagnates at a slightly lower level.
The difference is the most apparent for the highest considered choice~$k=4$.

\begin{figure}
    \centering
    \begin{subfigure}[t]{.48\textwidth}
        \centering
        \scalebox{0.55}{\input{figures/doubly_nonlinear_parabolic_p1.5_q1.5_nx50_energybalance.pgf}}
        \vspace{4pt}
        \caption{$p=1.5$, $q=1.5$}
    \end{subfigure}
    \hfill
    \begin{subfigure}[t]{.48\textwidth}
        \centering
        \scalebox{0.55}{\input{figures/doubly_nonlinear_parabolic_p1.5_q3_nx50_energybalance.pgf}}
        \vspace{4pt}
        \caption{$p=1.5$, $q=3$}
    \end{subfigure}
    \caption{Relative error in the energy balance for the full-order model of the doubly nonlinear parabolic equation~\eqref{eq:doubly-parabolic} for~$p=1.5$,~$q \in \{1.5, 3\}$ and various choices of~$k$ and~$\projnodes$.
        For the spatial discretization, we use~$h=\tfrac{1}{50}$.}
    \label{fig:doubly_nonlinear_parabolic_energybalance_fom}
\end{figure}

\begin{figure}
    \centering
    \begin{subfigure}[t]{.48\textwidth}
        \centering
        \scalebox{0.55}{\input{figures/doubly_nonlinear_parabolic_p1.5_q1.5_nx25_varying_degree.pgf}}
        \vspace{4pt}
        \caption{$p=1.5$, $q=1.5$}

    \end{subfigure}
    \hfill
    \begin{subfigure}[t]{.48\textwidth}
        \centering
        \scalebox{0.55}{\input{figures/doubly_nonlinear_parabolic_p1.5_q3_nx25_varying_degree.pgf}}
        \vspace{4pt}
        \caption{$p=1.5$, $q=3$}
    \end{subfigure}
    \caption{Convergence for the doubly nonlinear parabolic equation~\eqref{eq:doubly-parabolic} for~$p=1.5$ and~$q\in\{1.5, 3\}$.
        For the spatial discretization, we use~$h=\tfrac{1}{25}$.}
    \label{fig:doubly_nonlinear_parabolic_convergence}
\end{figure}

\begin{figure}
    \centering
    \begin{subfigure}[t]{.48\textwidth}
        \centering
        \scalebox{0.55}{\input{figures/doubly_nonlinear_parabolic_p1.5_q1.5_nx25_varying_degree_no_projection.pgf}}
        \vspace{4pt}
        \caption{$p=1.5$, $q=1.5$}

    \end{subfigure}
    \hfill
    \begin{subfigure}[t]{.48\textwidth}
        \centering
        \scalebox{0.55}{\input{figures/doubly_nonlinear_parabolic_p1.5_q3_nx25_varying_degree_no_projection.pgf}}
        \vspace{4pt}
        \caption{$p=1.5$, $q=3$}
    \end{subfigure}
    \caption{Convergence without the projective correction $\proj_2^\tau$ for the doubly nonlinear parabolic equation~\eqref{eq:doubly-parabolic} for~$p=1.5$ and~$q\in\{1.5, 3\}$.
        For the spatial discretization, we use~$h=\tfrac{1}{25}$.}
    \label{fig:doubly_nonlinear_parabolic_convergence_no_projection}
\end{figure}

\section{Conclusions}\label{sec:conclusion}

This paper proposes a discretization scheme for a class of energy based models.
Among others, the considered model class encompasses gradient systems and port-Hamiltonian systems and allows for additional algebraic constraints.
The presented discretization scheme is based on a Petrov--Galerkin approach using piecewise polynomials in time.
It is suitable for spatial and temporal discretization as well as model order reduction and qualitatively preserves the energy balance of the continuous or full order models.
In numerical experiments with a nonlinear circuit model, the Cahn--Hilliard equation, and a doubly nonlinear parabolic equation, we observe that our scheme has arbitrary order of convergence and that solutions satisfy the proposed energy balance up to machine precision in all time steps.

Topics for future research include studying the well-posedness of the scheme as well as providing proofs for the observed orders of convergence.
Further, in view of model order reduction, extending our ideas to nonlinear ansatz spaces and investigating hyperreduction techniques required for efficient online simulation of the reduced order model are interesting open research endeavors.

\section*{Code availability}
All custom code used to generate the results reported in this paper is available at \url{https://github.com/akarsai/structured-discretization-energy-based-models}.

\section*{Acknowledgments}
The authors thank T.~Reis for providing the matrices used for an older version of the nonlinear circuit example.
Furthermore, RA acknowledges support by the Deutsche For\-schungs\-ge\-mein\-schaft (DFG, German Research Foundation) - 446856041.
Moreover, AK thanks the Deutsche For\-schungs\-ge\-mein\-schaft for their support within the subproject B03 in the Son\-der\-for\-schungs\-be\-reich/Trans\-re\-gio 154 ``Mathematical Modelling, Simulation and Optimization using the Example of Gas Networks'' (Project 239904186).

\bibliographystyle{alpha}
\bibliography{references}

\end{document}